\documentclass{amsart}
\usepackage{amsmath,amssymb,latexsym}
\usepackage{amscd,amsthm}

\usepackage{tikz-cd}

\usepackage[all]{xy}

\newtheorem{theorem}{Theorem}[section]
\newtheorem{lemma}[theorem]{Lemma}
\newtheorem{proposition}[theorem]{Proposition}
\newtheorem{corollary}[theorem]{Corollary}

\newtheorem{notation}[theorem]{Notation}

\newtheorem{definition}{Definition}

\newtheorem{construction}[definition]{Construction}

\providecommand{\Ext}{\mathop{\rm Ext}\nolimits}
\providecommand{\Hom}{\mathop{\rm Hom}\nolimits}
\providecommand{\Tor}{\mathop{\rm Tor}\nolimits}

\newcommand{\cat}[1]{\mathcal{#1}}           
\newcommand{\cha}[1]{\textnormal{Ch}(\mathcal{#1})}

\newcommand{\tensor}{\otimes}

\newcommand{\class}[1]{\mathcal{#1}}   

\newcommand{\Z}{\mathbb{Z}}
\newcommand{\Q}{\mathbb{Q/Z}}

\newcommand{\ch}{\textnormal{Ch}(R)}

\newcommand{\tilclass}[1]{\widetilde{\class{#1}}}
\newcommand{\dwclass}[1]{dw\widetilde{\class{#1}}}
\newcommand{\exclass}[1]{ex\widetilde{\class{#1}}}
\newcommand{\dgclass}[1]{dg\widetilde{\class{#1}}}

\newcommand{\rightperp}[1]{#1^{\perp}}
\newcommand{\leftperp}[1]{{}^\perp #1}

\def\Qco{\mathfrak{Qco}}

\begin{document}

\title{Hereditary abelian model categories}

\author{James Gillespie}
\address{Ramapo College of New Jersey \\
         School of Theoretical and Applied Science \\
         505 Ramapo Valley Road \\
         Mahwah, NJ 07430}
\email[Jim Gillespie]{jgillesp@ramapo.edu}
\urladdr{http://pages.ramapo.edu/~jgillesp/}

\date{\today}

\begin{abstract}
We discuss some recent developments in the theory of abelian model categories. The emphasis is on the hereditary condition and applications to homotopy categories of chain complexes and stable module categories.
\end{abstract}

\maketitle

This paper is a survey of recent advances in the theory of abelian model categories, a tool of increasing importance in modern homological algebra. We give an updated account of the essentials of the theory and describe some recent exciting applications. These include their connection to triangulated categories, applications to homotopy categories of chain complexes and stable module categories, and recollement of triangulated categories.

Model categories were introduced by Quillen in~\cite{quillen-model categoires} and are most easily motivated by their utility in solving a localization problem for categories. Given a category $\cat{C}$ we may wish to treat the morphisms in a particular class $\class{W}$ as isomorphisms. For instance, $\cat{C}$ may be topological spaces and $\class{W}$ may be the class of weak homotopy equivalences, or, $\cat{C}$ may be a category of chain complexes and $\class{W}$ the class of homology isomorphisms. One may formally invert the morphisms of $\class{W}$ to obtain a potential category $\cat{C}[\class{W}^{-1}]$. But doing this leaves a category that is unwieldy and difficult to understand, for the morphisms are formal ``zig-zags'' of morphisms in $\cat{C}$. Quillen's theory solves this problem in the following way. By making some extra assumptions on the category $\cat{C}$, which are compatible with the class $\class{W}$ of \emph{weak equivalences}, we can introduce a formal relation on certain morphisms of $\cat{C}$, called the \emph{homotopy relation}. This leads to the construction of a new category $\textnormal{Ho}(\cat{C})$, the \emph{homotopy category} of $\cat{C}$. The fundamental result is a canonical equivalence $\cat{C}[\class{W}^{-1}] \cong \textnormal{Ho}(\cat{C})$. In particular, the morphism classes consisting of the cumbersome zig-zags are isomorphic to morphism \emph{sets} (so they are not just proper classes) in $\cat{C}$, modulo the formal homotopy relation.

The theory of abelian model categories concerns the case of when $\cat{C}$ is an abelian category and there is a model structure on $\cat{C}$ that is compatible with the abelian structure. They were introduced by Hovey, in~\cite{hovey}, and the cornerstone of the theory is \emph{Hovey's correspondence}: A one-to-one correspondence between abelian model structures and certain (pairs of) cotorsion pairs. Since model categories were invented and are mainly used by topologists, this correspondence serves as a translator between ideas rooted in topology and ideas rooted in algebra. More precisely, the translation is between ideas from homotopy theory and ideas from homological algebra.

Hovey showed that in the case of abelian model categories, the class $\class{W}$ of weak equivalences is completely determined by a class of objects, called the \emph{trivial objects}; see Proposition~\ref{prop-characterization of weak equivs}. Under a mild hypothesis, called the \emph{hereditary} condition, the homotopy category $\textnormal{Ho}(\cat{C})$ will also inherit the structure of a triangulated category. In this case we may think of $\textnormal{Ho}(\cat{C})$ as the triangulated category obtained from killing the objects in $\class{W}$. The hereditary condition is present in virtually every abelian model structure encountered in practice. This paper focuses on constructing triangulated categories by way of building hereditary model structures. New methods have made such constructions quite easy, and we illustrate this by generating several models that have recently appeared in exciting applications. For instance, Stovicek recently used abelian model categories to show that $K(Inj)$, the chain homotopy category of all chain complexes of injective $R$-modules, is a compactly generated triangulated category whenever $R$ is a coherent ring~\cite{stovicek-purity}. In~\cite{bravo-gillespie-hovey}, abelian model categories were used to introduce the stable module category of a general ring. And in~\cite{becker,gillespie-recollements}, there appeared an elegant connection between abelian model categories and recollement of triangulated categories.

Ideally, the reader will have some familiarity with homological algebra but perhaps is not so familiar with model categories. On the other hand, folks with a general interest in algebraic topology may be interested to see an updated account on the current state of abelian model categories. In any case, the intention is to give the general reader an introduction to the subject with complete references for further study.

\section{Notation and Conventions}\label{sec-notation}

Throughout the paper $R$ denotes a ring. We do not assume it is commutative, and unless stated otherwise, all modules are assumed to be left $R$-modules. We will use the following notation throughout the paper. One can refer to a standard reference such as~\cite{weibel} for further details on $K(R)$ and $\class{D}(R)$.

\

\noindent $R$-Mod denotes the category of (left) $R$-Modules.

\

\noindent $\ch$ denotes the category of chain complexes of $R$-modules. Its objects are chain complexes $\cdots
\xrightarrow{} X_{n+1} \xrightarrow{d_{n+1}} X_{n} \xrightarrow{d_n}
X_{n-1} \xrightarrow{} \cdots$ and its morphisms are chain maps.

\

\noindent $K(R)$ denotes the homotopy category of chain complexes. Its objects are also chain complexes but its morphisms are homotopy classes of chain maps.

\

\noindent $\class{D}(R)$ denotes the derived category of $R$. It is the Verdier quotient of $K(R)$ by the class $\class{E}$ of all exact (acyclic) chain complexes.

\

\noindent $S^n(M)$ is the \emph{$n$-sphere} on a given $R$-module $M$. This is the chain complex consisting only of $M$ in degree $n$ and 0 elsewhere.

\

\noindent $D^n(M)$ is the \emph{$n$-disk} on a given $R$-module $M$. It is the chain complex consisting only of $M \xrightarrow{1_M} M$ concentrated in degrees $n$ and $n-1$, and 0 elsewhere.

\

\noindent $\Sigma^n X$ is the \emph{$n^{\text{th}}$ suspension} of a given chain complex $X$. It is the chain complex given by $(\Sigma^n X)_{k} = X_{k-n}$ and $(d_{\Sigma^n X})_{k} = (-1)^nd_{k-n}$.

\section{Fundamentals of abelian model categories}\label{section-fundamentals}

Abelian model categories were introduced by Mark Hovey in~\cite{hovey} as abelian categories possessing a compatible model structure in the sense of Quillen~\cite{quillen-model categoires}. The central result is now known as \emph{Hovey's correspondence}: A correspondence between complete cotorsion pairs and abelian model structures. In this section we describe this correspondence and some first fundamentals of abelian model categories.  Further theory will be developed in Section~\ref{sec-hereditary abelian models} after two motivating examples are given in Section~\ref{sec-motivation examples}. Throughout this section, $\cat{A}$ denotes an abelian category.

\subsection{Complete cotorsion pairs}\label{sec-cotorsion pairs} Cotorsion pairs were introduced by Luigi Salce in~\cite{salce} in the context of abelian groups. We think of a cotorsion pair in $\cat{A}$ as a pair of classes that are orthogonal with respect to $\Ext^1_{\cat{A}}(-,-)$. Regardless of whether or not $\cat{A}$ has enough projectives or injectives, recall that the functors $\Ext^n_{\cat{A}}$ can be defined as equivalence classes of exact sequences, and made into an abelian group under the Baer sum~\cite[Vista~3.4.6]{weibel}. In particular, $\Ext^1_{\cat{A}}(X,Y) = 0$ if and only if every short exact sequence $0 \xrightarrow{} Y \xrightarrow{} Z \xrightarrow{} X \xrightarrow{} 0$ splits.

\begin{definition}
A \textbf{cotorsion pair} in $\cat{A}$ is a pair of classes $(\class{X},\class{Y})$ of objects in $\cat{A}$ satisfying the following two conditions:
\begin{enumerate}
    \item $X \in \class{X}$ iff $\Ext^1_{\cat{A}}(X,Y) = 0$ for all $Y \in
    \class{Y}$.

    \item $Y \in \class{Y}$ iff $\Ext^1_{\cat{A}}(X,Y) = 0$ for all $X \in
    \class{X}$.
\end{enumerate}
\end{definition}

The cotorsion pair is called \textbf{complete} if for any $A \in \cat{A}$ there exists a short exact sequence $0 \xrightarrow{} Y \xrightarrow{} X \xrightarrow{} A \xrightarrow{}0$ with $X \in \class{X}$ and $Y \in \class{Y}$, and another short exact sequence $0\xrightarrow{} A \xrightarrow{} Y' \xrightarrow{}X'\xrightarrow{}0$ with $X' \in \class{X}$ and $Y' \in \class{Y}$. The first sequence generalizes the concept of having enough projectives while the second generalizes the concept of having enough injectives. Indeed letting $\class{I}$ and $\class{P}$ denote the class of injective and projective objects, we have the \textbf{categorical cotorsion pairs} $(\class{A},\class{I})$ and $(\class{P},\class{A})$. The former is complete precisely when $\cat{A}$ has enough injectives and the second is complete when $\cat{A}$ has enough projectives.
The canonical nontrivial example of a complete cotorsion pair is Ed Enochs' flat cotorsion pair $(\class{F},\class{C})$, in the category of modules over a ring $R$. Here $\class{F}$ is the class of \emph{flat modules}; that is, those modules $F$ for which $- \otimes_R F$ is an exact functor. The modules in $\class{C}$ are called \emph{cotorsion modules}. A standard reference for all of this and related concepts is the book~\cite{enochs-jenda-book}.

A well known and powerful method for constructing complete cotorsion pairs is to ``cogenerate'' one from a set. Given any set $\class{S}$ of objects in $\cat{A}$, we let $\rightperp{\class{S}}$ be the class of all objects $Y$ such that $\Ext^1_{\cat{A}}(S,Y) = 0$ for all $S \in \class{S}$. A cotorsion pair $(\class{X},\class{Y})$ is said to be \textbf{cogenerated by a set} if there exists a set $\class{S}$ such that $\rightperp{\class{S}} = \class{Y}$.  In practice, complete cotorsion pairs almost always arise as an application of the following theorem.  Recall that a \textbf{Grothendieck category} is an abelian category possessing a generator and such that direct limits of short exact sequences are again short exact sequences. For example, $R$-Mod is a Grothendieck category. Viewing $R$ as a (left) module over itself provides a generator for $R$-Mod; that is, every module $M$ is an epimorphic image of a direct sum of copies of $R$.

\begin{theorem}\label{them-cogen-by-a-set}
Let $\cat{A}$ be a Grothendieck category and let $\class{S}$ be any set containing a generator. Let $\class{Y} = \rightperp{\class{S}}$. Then the cotorsion pair cogenerated by $\class{S}$, that is $(\leftperp{\class{Y}},\class{Y})$, is a complete cotorsion pair.
\end{theorem}

Theorem~\ref{them-cogen-by-a-set} received a lot of attention after it was proved in the context of $R$-modules by Eklof and Trlifaj~\cite[Them.~10]{eklof-trlifaj}, for Enochs then used it to settle his \emph{flat cover conjecture}. It states that any $R$-module has a flat cover, a notion which can be thought of as dual to an injective hull. The version we give above is from~\cite[Section~2]{saorin-stovicek}. The essential ideas even go back to Quillen, and are an instance of his \emph{small object argument}. This was made clear by Hovey in~\cite[Them.~6.5]{hovey}, and by Saorin-Stovicek in~\cite[Corollary~2.14]{saorin-stovicek}. In Hovey's correspondence between cotorsion pairs and abelian model structures, which we describe next, cotorsion pairs that are cogenerated by a set correspond to \emph{cofibrantly generated model structures}. See~\cite[Section~2.1]{hovey-model-categories}, \cite[Section~6]{hovey}, and~\cite[Section~2]{saorin-stovicek}.

\subsection{Abelian model categories and Hovey's correspondence}

Apart from the algebraic notion of a cotorsion pair, there is the topological notion of a model category introduced by Quillen in~\cite{quillen-model categoires}.  Simply put, a model category is a category $\mathcal{M}$ along with three distinguished subclasses of maps called \emph{cofibrations}, \emph{fibrations} and \emph{weak equivalences}, all of which satisfy several axioms. The point of the axioms is to introduce a homotopy theory into the category. The standard references include~\cite{quillen-model categoires,dwyer-spalinski,hovey-model-categories,hirschhorn-model-categories}. We will not repeat here the (long) definition of a model category, simply because we won't need it. Indeed the success of Hovey's correspondence is that it translates these axioms into much simpler statements about cotorsion pairs. In particular, there is even no harm if the reader wishes to take the \emph{definition} of an abelian model category to be a \emph{Hovey triple} as in Theorem~\ref{them-hovey-correspondence}.

Hovey observed that the model structures appearing on abelian categories in nature satisfy some compatibilities between the abelian structure and the model structure. This of course is something that we expect in mathematics whenever we merge two ideas. To explain in this case, assume an abelian category $\cat{A}$ has a model structure in the sense of~\cite[Definition~1.1.3]{hovey-model-categories} and let $X \in \cat{A}$ be an object. We say that $X$ is \textbf{trivial} if $0 \xrightarrow{} X$ is a weak equivalence; we think of these objects as the ones to be ``killed'', or forced to be zero. We say $X$ is \textbf{(trivially) cofibrant} if $0 \xrightarrow{} X$ is a (trivial) cofibration. We say that $X$ is \textbf{(trivially) fibrant} if $X \xrightarrow{} 0$ is a (trivial) fibration. With this language, for a model structure on an abelian category to qualify as abelian, we require that the two structures be compatible in the following sense.

\begin{definition}\label{def-abelian}
Let $\cat{A}$ be a bicomplete abelian category with a model structure in the sense of Definition~1.1.3 of~\cite{hovey-model-categories}. We call it an \textbf{abelian model structure} if each of the following holds.
\begin{enumerate}
\item A morphism $f$ is a (trivial) cofibration if and only if it is a monomorphism with (trivially) cofibrant cokernel.
\item A morphism $g$ is a (trivial) fibration if and only if it is an epimorphism with (trivially) fibrant kernel.
\end{enumerate}
\end{definition}

As shown in~\cite{hovey}, this definition is stronger than it needs to be. However, the above definition is equivalent and makes it clear that we can now shift all our attention  from morphisms (cofibrations, weak equivalences, and fibrations) to objects (cofibrant, trivial, and fibrant). This is the key to the relative simplicity of abelian model categories as opposed to general model categories. This brings us to Hovey's correspondence, which essentially says that an abelian model structure is just three classes of objects, the cofibrant, trivial, and fibrant objects, all satisfying some properties. To state the correspondence, call a class of objects $\class{W}$ \textbf{thick} if it is closed under direct summands and satisfies that whenever two out of three terms in a short exact sequence are in $\class{W}$, then so is the third.

\begin{theorem}[Hovey's Correspondence {\cite[Theorem~2.2]{hovey}}]\label{them-hovey-correspondence}
Let $\class{A}$ be an abelian category with an abelian model structure. Let the triple $(\class{Q},\class{W},\class{R})$ denote respectively the classes $\class{Q}$ of cofibrant objects, $\class{W}$ of trivial objects, and $\class{R}$ of fibrant objects. Then $\class{W}$ is a thick class in $\class{A}$ and both $(\class{Q},\class{W} \cap \class{R})$ and $(\class{Q} \cap \class{W} , \class{R})$ are complete cotorsion pairs in $\cat{A}$. Conversely, given a thick class $\class{W}$ and classes $\class{Q}$ and $\class{R}$ making $(\class{Q},\class{W} \cap \class{R})$ and $(\class{Q} \cap \class{W} , \class{R})$ each complete cotorsion pairs, then there is an abelian model structure on $\cat{A}$ where $\class{Q}$ are the cofibrant objects, $\class{R}$ are the fibrant objects and $\class{W}$ are the trivial objects.
\end{theorem}

Hovey's correspondence makes it clear that an abelian model structure can be succinctly represented by a triple $\class{M} = (\class{Q},\class{W},\class{R})$. By a slight abuse of language we often refer to such a triple as an \emph{abelian model structure}. Alternately, we often call $\class{M}$ a \textbf{Hovey triple}.

\subsection{Characterization of weak equivalences and trivial objects}\label{sec-trivial objects}

The weak equivalences are the most important class of morphisms in a model category, for they determine its homotopy category, as explained below in the Fundamental Theorem~\ref{them-fundamental}. However,  Theorem~\ref{them-hovey-correspondence}/Definition~\ref{def-abelian} do not explicitly characterize the weak equivalences. As it turns out, applying the so called \emph{two out of three axiom} for model categories, it is easy to see that a map is a weak equivalence if and only if it factors as a trivial cofibration followed by a trivial fibration. In fact we can do better than that. The following important characterization is proved in~\cite[Lemma~5.8]{hovey}.

\begin{proposition}\label{prop-characterization of weak equivs}
Let $\class{M} = (\class{Q}, \class{W}, \class{R})$ be a Hovey triple in $\class{A}$. Then a morphism is a weak equivalence if and only if it factors as a monomorphism with trivial cokernel followed by an epimorphism with trivial kernel.
\end{proposition}

So just as the property of being a cofibration (or resp. fibration) is completely determined by the class $\class{Q}$ (resp. $\class{R}$), Proposition~\ref{prop-characterization of trivial objects} tells us that the weak equivalences are completely determined by $\class{W}$. This is consistent with what we already stated above: In abelian model categories, all of the defining concepts are moved from morphisms to objects, making them much easier than general model categories. Furthermore, we have the following characterization of the trivial objects.

\begin{proposition}\label{prop-characterization of trivial objects}
Let $\class{M} = (\class{Q}, \class{W}, \class{R})$ be a Hovey triple in $\class{A}$.
Then the thick class $\class{W}$ is characterized in each of the two following ways:
  \begin{align*}
   \class{W}  &= \{\, A \in \class{A} \, | \, \exists \, \text{s.e.s.} \ 0 \xrightarrow{} A \xrightarrow{} V \xrightarrow{} W \xrightarrow{} 0 \, \text{ with} \, V \in \class{W} \cap \class{R} \, , W \in \class{Q} \cap \class{W} \,\} \\
           &= \{\, A \in \class{A} \, | \, \exists \, \text{s.e.s.} \ 0 \xrightarrow{} V' \xrightarrow{} W' \xrightarrow{} A \xrightarrow{} 0 \, \text{ with} \,  V' \in \class{W} \cap \class{R} \, , W' \in \class{Q} \cap \class{W} \,\}.
          \end{align*}
Consequently, $\class{W}$ is unique in the sense that whenever $(\class{Q}, \class{V}, \class{R})$ is a Hovey triple, then necessarily $\class{V} = \class{W}$.
\end{proposition}

The easy proof of the above proposition can be found in~\cite[Proposition~3.2]{gillespie-recollements}. Note that it implies $\class{W}$ is the smallest thick class containing both the trivially cofibrant and trivially fibrant objects.

\subsection{Cofibrant and fibrant replacement} \label{sec-replacements}

Hovey's Correspondence Theorem~\ref{them-hovey-correspondence} provides a method to build a dictionary between model category ideas and properties of cotorsion pairs.
Now all model categories come with the fundamental notion of cofibrant replacement and fibrant replacement. We now interpret these for abelian model categories.

The idea of cofibrant replacement is to replace an object $X$ by a cofibrant object $Q$, which is isomorphic to $X$ in the homotopy category $\textnormal{Ho}(\class{M})$. Formally, a \textbf{cofibrant replacement} of $X$ is any factorization of the map $0 \xrightarrow{} X$, where $0$ is the initial object,  as a cofibration $0 \xrightarrow{} Q$ followed by a trivial fibration $Q \xrightarrow{p} X$. In particular, $Q$ is cofibrant and $p$ is a weak equivalence.
On the other hand, a \textbf{fibrant replacement} of $X$ is a factorization of the map $X \xrightarrow{} *$, where $*$ is the terminal object, as a trivial cofibration $X \xrightarrow{i} R$ followed by a fibration $R \xrightarrow{} *$. Hence $R$ is fibrant and $i$ is a weak equivalence. Note that any abelian category is pointed in the sense that $0 \xrightarrow{} *$ is an isomorphism; we just have a zero object $0$. The following lemma now becomes obvious and provides the standard mechanism for taking cofibrant and fibrant replacements in abelian model categories.

\begin{lemma}[Cofibrant and fibrant replacement]\label{lemma-cof-replacements}
Let $\class{M} = (\class{Q}, \class{W}, \class{R})$ be a Hovey triple in $\class{A}$ and let $X \in \cat{A}$ be any object.
\begin{enumerate}
\item Using completeness of the cotorsion pair $(\class{Q},\class{W} \cap \class{R})$, write a short exact sequence $0 \xrightarrow{} W \xrightarrow{} QX \xrightarrow{p_X} X \xrightarrow{} 0$ with $QX \in\class{Q}$ and $W \in \class{W} \cap \class{R}$. Then $QX$ is a cofibrant replacement of $X$ in $\class{M}$.
\item Using completeness of the cotorsion pair  $(\class{Q} \cap \class{W}, \class{R})$, write a short exact sequence  $0 \xrightarrow{} X \xrightarrow{i_X} RX \xrightarrow{} W \xrightarrow{} 0$ with $RX \in\class{R}$ and $W \in \class{Q} \cap \class{W}$. Then $RX$ is a fibrant replacement of $X$ in $\class{M}$.
\end{enumerate}
\end{lemma}
Finally, given $A \in \cat{A}$, by a \textbf{bifibrant replacement} of $X$ we mean $RQX$.

\subsection{The homotopy category of an abelian model category} \label{sec-fundamental theorem}

We end this section by describing Quillen's classical result, on how model categories are used to represent localization categories $\cat{C}[\class{W}^{-1}]$. Every standard reference on model categories contains this material. So we keep this brief, emphasizing what the results says for \emph{abelian} model categories. The proof will just provide references to the literature for full details.

First, let $\class{W}$ be a class of morphisms in a category $\cat{C}$. One can formally invert the morphisms of $\class{W}$ to get a ``category'' $\cat{C}[\class{W}^{-1}]$, where the morphisms of $\class{W}$ have been forced to become isomorphisms. See~\cite[Def.~1.2.1]{hovey-model-categories}. One obtains a canonical functor $\gamma : \cat{C} \xrightarrow{} \cat{C}[\class{W}^{-1}]$ which is universally initial with respect to the property of inverting the morphisms of $\class{W}$. But this ``category'' may not be a category at all, in the sense that morphism sets are not guaranteed to be actual sets, just proper classes. The standard result is that if $\class{W}$ is the class of weak equivalences in a model structure on $\cat{C}$, then $\cat{C}[\class{W}^{-1}]$ is indeed a category and can be represented (i.e., ``modeled'') via the model structure. The essentials are summarized in the following fundamental theorem.

\begin{theorem}[The Fundamental Theorem of Model Categories]\label{them-fundamental}
Suppose $\class{M} = (\class{Q},\class{W}, \class{R})$ is an abelian model category. Then there is a category, $\textnormal{Ho}(\class{M})$, called the \textbf{homotopy category of $\class{M}$}, whose objects are the same as those of $\cat{A}$ but whose morphisms are given by $$\Hom_{\textnormal{Ho}(\class{M})}(X,Y) = \Hom_{\cat{A}}(RQX,RQY)/\sim.$$
Here, $RQX$ is a chosen bifibrant replacement of $X$ as in Lemma~\ref{lemma-cof-replacements}, and $f \sim g$ if and only if $g-f$ factors through an object of the core $\class{Q}\cap\class{W}\cap \class{R}$. Moreover, we have:
\begin{enumerate}
\item The inclusion $\class{Q} \cap \class{R} \hookrightarrow \cat{A}$ induces an equivalence of categories $$(\class{Q} \cap \class{R})/\sim \ \hookrightarrow \textnormal{Ho}(\class{M}),$$ where again $f \sim g$ if and only if $g-f$ factors through an object of the core $\class{Q}\cap\class{W}\cap \class{R}$. Its inverse is obtained by taking bifibrant replacements.

\item There is a functor $\gamma : \cat{A} \xrightarrow{} \textnormal{Ho}(\class{M})$ which is the identity on objects and bifibrant replacement on morphisms. It is a localization of $\cat{A}$ with respect to $\class{W}$, and hence there is a canonical equivalence of categories $\cat{C}[\class{W}^{-1}] \cong \textnormal{Ho}(\class{M})$.

\item For any choice of cofibrant replacement $QX$ and fibrant replacement $RY$, we have $\Hom_{\text{Ho}(\cat{A})}(X,Y) \cong \Hom_{\cat{A}}(QX,RY)/\sim$.
\end{enumerate}
\end{theorem}

\noindent \emph{References for proof}: Again, this goes back to~\cite{quillen-model categoires} and can be found in most standard references on model categories. The key difference in our statement above is the given characterization of the homotopy relation. As stated, it is only valid for \emph{abelian} model categories, and this was proved in~\cite[Prop.~4.4]{gillespie-exact model structures}. The particular method used for defining the homotopy category in the theorem is based on~\cite[Definition~5.6/5.7]{dwyer-spalinski}. The reader will find detailed proofs of the above statements in~\cite[Sections~5 and~6]{dwyer-spalinski}. Another good alternative is to consult~\cite[Theorem~1.2.10]{hovey-model-categories}.

\section{Two motivating examples}\label{sec-motivation examples}

In this section we describe the main motivating examples of abelian model structures. They are the two standard models on $\ch$ for the derived category, and two models on $R$-Mod for the stable module category of a Gorenstein ring $R$.

\subsection{Models for the derived category of a ring} \label{sec-derived-models}
We can think of $\class{D}(R)$, the derived category of $R$,  as the triangulated category obtained from $\ch$ by killing the exact (or acyclic) chain complexes. By a fundamental result of homological algebra, the class  $\class{W}$ of all exact complexes is indeed thick. So to model $\class{D}(R)$, one seeks an abelian model structure on $\ch$ with $\class{W}$ as the class of trivial objects. The most common solutions to this are the \textbf{injective model structure} and the \textbf{projective model structure} on $\ch$.  The injective model structure was written down in~\cite[Section~2.3]{hovey-model-categories} before the correspondence between abelian model structures and cotorsion pairs was realized. From the cotorsion pair perspective, it can easily be represented by the Hovey triple $\class{M}_{inj} = (\class{A}, \class{W},\class{I})$. Here $\class{A}$ denotes the class of all chain complexes while $\class{I}$ is the class of all DG-injective complexes~\cite{avramov-foxby-halperin}. A chain complex $I$ is \emph{DG-injective} if each component $I_n$ is injective and all chain maps $E \xrightarrow{} I$ are null homotopic whenever $E$ is an exact complex. From~\cite{enochs-jenda-xu-97} we see that  $(\class{W},\class{I})$ is a complete cotorsion pair with $\class{W} \cap \class{I}$ coinciding with the class of all categorically injective complexes. Hence $\class{M}_{inj}$ is a Hovey triple.

On the other hand, the projective model structure on $\ch$ is represented by the Hovey triple $\class{M}_{prj} = (\class{P}, \class{W},\class{A})$. Here $\class{P}$ is the class of all DG-projective chain complexes. They are the complexes $P$ with each component $P_n$ projective and all chain maps $P \xrightarrow{} E$ being null homotopic whenever $E$ is exact. For more details, we again refer to~\cite{avramov-foxby-halperin,enochs-jenda-xu-97,hovey-model-categories}. Also, the monograph~\cite{garcia-rozas} is a standard reference for DG-injective and DG-projective complexes.

It is enlightening to interpret the Fundamental Theorem~\ref{them-fundamental} with these model structures. Recall that $\Ext^n_R(M,N)$ is classically computed by taking either a projective resolution of $M$ or an injective resolution of $N$. Following the projective approach, take a projective resolution of $M$:
$$\cdots \xrightarrow{} P_2 \xrightarrow{} P_1 \xrightarrow{} P_0 \xrightarrow{\epsilon} M  \xrightarrow{}  0.$$
For brevity, we denote this by $P_* \xrightarrow{\epsilon} M \xrightarrow{} 0$, where $P_*$ is the complex $\cdots \xrightarrow{} P_2 \xrightarrow{} P_1 \xrightarrow{} P_0 \xrightarrow{}  0$. Then there is an obvious short exact sequence $0 \xrightarrow{} K \xrightarrow{}  P_* \xrightarrow{\epsilon} S^0(M) \xrightarrow{} 0$. Since $P_*$ is bounded below, one can prove by induction that it is a DG-projective complex. Moreover, $K$ is easily seen to be exact. Hence $P_*$ is a cofibrant replacement of $S^0(M)$ in $\class{M}_{prj}$, by Lemma~\ref{lemma-cof-replacements}. Now applying the Fundamental Theorem~\ref{them-fundamental} (iii), we get $$\textnormal{Ho}(\class{M}_{prj})(S^0(M),S^n(N)) \cong \ch(P_*,S^n(N))/\sim$$
and $\sim$ turns out to be the usual chain homotopy relation by Lemma~\ref{lemma-chain-homotopy}. On the other hand, $$\ch(P_*,S^n(N))/\sim \ \cong H^n[\Hom_R(P_*,N)] \cong \Ext^n_R(M,N).$$ We conclude $\textnormal{Ho}(\class{M}_{prj})(S^0(M),S^n(N)) \cong \Ext^n_R(M,N)$, and that the projective model structure formalizes the classical computation of $\Ext$ via projective resolutions. A similar argument applies to the injective model structure; it corresponds to the existence of injective resolutions for computing $\Ext^n_R(M,N)$.

\subsection{Models for the stable module category of a Gorenstein ring}\label{section-stable module cat}

Consider a field $k$ and a finite group $G$. Then $k[G]$ is a quasi-Frobenius ring, meaning it is both left and right Noetherian and the projective modules coincide with the injective modules. In representation theory, the stable module category St($k[G]$) naturally arises from $k[G]$-Mod by killing the projective-injective modules. Formally, the objects of St($k[G]$) are just $k[G]$-modules, but morphisms are identified by $f \sim g$ whenever $g-f$ factors through a projective-injective module.  Analogous to how we saw above that $\Ext$ groups live as morphism sets in the derived category, the Tate cohomology of $G$ lives here as morphism sets in St($k[G]$). Stable module theory goes back to~\cite{auslander-bridger} and has been studied by many authors. In particular, a classical abstract homological treatment due to Buchweitz can be found online; see~\cite{buchweitz}.

The definition of St($k[G]$) is reminiscent of the formal $\sim$ in the Fundamental Theorem~\ref{them-fundamental}. So it is not surprising that there is an abelian model structure on $k[G]$-Mod whose homotopy category is St($k[G]$). In fact, one sees immediately that there is a Hovey triple $(\class{A},\class{W},\class{A})$, where $\class{W}$ is the class of projective-injective modules, and $\cat{A}$ denotes the class of all $k[G]$-modules. Using abelian model categories, Hovey generalized this construction to Gorenstein rings $R$. Such rings include quasi-Frobenius rings, and as we now explain, the trivial objects are relaxed from the projective-injective modules to include all modules of finite projective/injective dimension. Recall that a module $M$ has finite injective dimension if it has a finite injective resolution $$0 \xrightarrow{} M \xrightarrow{} I^0 \xrightarrow{} I^1 \xrightarrow{} \cdots \xrightarrow{} I^n \xrightarrow{} 0,$$ and finite projective dimension is defined similarly.

\begin{definition}
A ring $R$ is called \textbf{Gorenstein} if it is both left and right Noetherian and has finite injective dimension as both a left and right module over itself.
\end{definition}

Such rings were introduced and studied in~\cite{iwanaga,iwanaga2}. While Gorenstein rings have long been important in commutative algebra, this definition extends the notion to non commutative rings. If $K$ is a commutative Gorenstein ring, then the group ring $K[G]$ is Gorenstein for any finite group $G$~\cite{eilenberg-nak}. Iwanaga's main result on Gorenstein rings is that if $R$ is Gorenstein with injective dimension $\textnormal{id}({}_RR) = n < \infty$, then the class of $R$-modules of finite injective dimension coincides with the class of $R$-modules of finite projective dimension. Furthermore, $n$ is an upper bound for any module of finite projective or injective dimension. See~\cite[Section~9.1]{enochs-jenda-book}.

Let $\class{W}$ denote the class of all modules of finite injective (and hence projective) dimension. In light of the above it is easily seen that $\class{W}$ is thick. We set $\class{GP} = \leftperp{\class{W}}$; the modules in $\class{GP}$ are called \emph{Gorenstein projective}. These modules have been studied by many people with much interest initially given by Enochs, Jenda, and Xu. On the other hand, set $\class{GI} = \rightperp{\class{W}}$. The modules in $\class{GI}$ are called \emph{Gorenstein injective}. The standard reference for Gorenstein injective and Gorenstein projective modules is~\cite[Chapter~10]{enochs-jenda-book}.  In the seminal paper~\cite{hovey}, Hovey defines the stable module category St($R$) of a Gorenstein ring, by putting two abelian model structures on $R$-Mod, each having $\class{W}$ as the class of trivial objects.

\begin{theorem}[\cite{hovey}, Theorem~2.6]\label{them-Goren-models}
Let $R$ be a Gorenstein ring. Then there are two abelian model structures on $R$-Mod as follows:
The injective model structure $\class{M}_{inj} = (All, \class{W},\class{GI})$ has all modules cofibrant and the fibrant objects are the Gorenstein injective modules. The projective model structure $\class{M}_{prj} = (\class{GP},\class{W}, All)$ has all modules fibrant and the Gorenstein projective modules are cofibrant. Since they share the same class of trivial objects, their homotopy categories are equivalent. We denote it by $\textnormal{St}(R)$ and call it the \textbf{stable module category of $R$}. Referring to the Fundamental Theorem~\ref{them-fundamental} we have equivalences $$\class{GP}/\sim \ \hookrightarrow  \textnormal{St}(R)  \hookleftarrow \class{GI}/\sim,$$
where $f \sim g$ if and only if $g-f$ factors through a module of finite projective/injective dimension.
\end{theorem}

\section{Hereditary abelian model categories}\label{sec-hereditary abelian models}

We now discuss what has turned out to be a crucial property of cotorsion pairs and abelian model categories, the \emph{hereditary} property. We see this property in virtually every cotorsion pair we encounter and it is connected to the theory of abelian model categories in at least three important ways. First, Theorem~\ref{them-how to construct triples} supplements Hovey's correspondence, making it easier than ever to construct abelian model structures from hereditary cotorsion pairs. Second, each abelian model category $\class{M}$ constructed in this way is stable; that is, its homotopy category $\text{Ho}(\class{M})$ is naturally equivalent to the stable category of a Frobenius category.  Finally, the hereditary hypothesis on cotorsion pairs has been crucial to the problem of lifting complete cotorsion pairs to models for the derived category. This last point will be discussed in Section~\ref{sec-chain-complexes}.

A cotorsion pair $(\class{X},\class{Y})$ is called \textbf{hereditary} if $\class{X}$ is closed under taking kernels of epimorphisms between objects in $\class{X}$ and $\class{Y}$ is closed under taking cokernels of monomorphisms between objects in $\class{Y}$. We also call an abelian model structure \emph{hereditary} when the two associated cotorsion pairs in Hovey's correspondence Theorem~\ref{them-hovey-correspondence} are hereditary. All of the cotorsion pairs encountered in this paper will be hereditary, reflecting that virtually all the cotorsion pairs and abelian model structures we encounter in practice tend to be hereditary.

\subsection{Construction of hereditary abelian model categories}
We now describe a theorem which provides a convenient method to quickly construct hereditary abelian model structures. Its utility will be illustrated throughout the rest of the paper, to give easy proofs of the existence of many model structures with triangulated homotopy categories.

Note that Hovey's correspondence is between triples $(\class{Q},\class{W},\class{R})$ with $\class{W}$ a thick class for which $(\class{Q} , \class{W} \cap \class{R})$ and $(\class{Q} \cap \class{W},\class{R})$ are complete cotorsion pairs. But it turns out that a thick class $\class{W}$ can be constructed from two complete cotorsion pairs in the hereditary case.  To state the result, given a cotorsion pair $(\class{X},\class{Y})$, call $\class{X} \cap \class{Y}$ the \textbf{core} of the cotorsion pair. Note that for a Hovey triple $(\class{Q},\class{W}, \class{R})$, the two associated cotorsion pairs $(\class{Q}, \tilclass{R}) = (\class{Q},\class{W} \cap \class{R})$ and $(\tilclass{Q}, \class{R}) = (\class{Q} \cap \class{W},\class{R})$ have the same core. This class, $\class{Q} \cap \class{W} \cap \class{R}$ is called the \emph{core} of the Hovey triple.

\begin{theorem}[\cite{gillespie-hovey triples}]\label{them-how to construct triples}
Let $(\class{Q}, \tilclass{R})$ and $(\tilclass{Q}, \class{R})$ be two complete hereditary cotorsion pairs with equal cores and with $\tilclass{R} \subseteq \class{R}$, or equivalently, $\tilclass{Q} \subseteq \class{Q}$. Then there is a unique thick class $\class{W}$ for which $(\class{Q},\class{W},\class{R})$ is a Hovey triple. It is the smallest thick class containing the trivially fibrant and trivially cofibrant objects and satisfies the following two characterizations:
\begin{align*}
   \class{W}  &= \{\, X \in \class{A} \, | \, \exists \, \text{s.e.s. } \, 0 \xrightarrow{} X \xrightarrow{} R \xrightarrow{} Q \xrightarrow{} 0 \, \text{ with} \, R \in \tilclass{R} \, , Q \in \tilclass{Q} \,\} \\
           &= \{\, X \in \class{A} \, | \, \exists \, \text{s.e.s. } \, 0 \xrightarrow{} R' \xrightarrow{} Q' \xrightarrow{} X \xrightarrow{} 0 \, \text{ with} \, R' \in \tilclass{R} \, , Q' \in \tilclass{Q} \,\}.
          \end{align*}
\end{theorem}

\begin{proof}
A self contained elementary proof is detailed in~\cite{gillespie-hovey triples}.
\end{proof}

It took quite a while for Theorem~\ref{them-how to construct triples} to be found. This is because $\class{W}$ is the most important class in a model structure since,  according to Proposition~\ref{prop-characterization of weak equivs} and Theorem~\ref{them-fundamental}, it determines the homotopy category. Naturally, one would typically already have in mind what $\class{W}$ should be when constructing an abelian model structure. But Theorem~\ref{them-how to construct triples} turns this idea on its head. For example, it has been known for quite some time that we have a complete hereditary cotorsion pair $(\class{Q}, \tilclass{R})$ in $\ch$ where $\class{Q}$  is the class of all complexes with flat components, and another complete hereditary cotorsion pair $(\tilclass{Q}, \class{R})$ where $\tilclass{Q}$ is the class of all categorically flat chain complexes. They satisfy the hypotheses of Theorem~\ref{them-how to construct triples}. So this raises a question: What is the corresponding model structure actually ``modeling''? We see in Section~\ref{sec-coderived-models} that this is a model for Murfet's \emph{mock homotopy category of projectives}, from~\cite{murfet-thesis}.
Similarly, we can find model structures on the entire chain complex category, $\ch$, that model such categories as $K(Inj)$, the chain homotopy category of all complexes with injective components. The reader can skip right away  to Section~\ref{sec-chain-complexes} to learn more about these and other examples.

\subsection{The triangulated structure on $\textnormal{Ho}(\class{M})$}\label{sec-triangulated}

An important feature of hereditary abelian model structures is that $\textnormal{Ho}(\class{M})$ is always a triangulated category, in the sense of Verdier, in the case that $\class{M}$ is hereditary. Standard examples of triangulated categories include $K(R)$ and $\class{D}(R)$ but also St($R$), the stable module category of a Gorenstein ring (Section~\ref{section-stable module cat}). It is fair to think of triangulated categories as additive categories which are not quite abelian, but which have ``exact triangles'' substituting for short exact sequences. The standard references include~\cite{weibel,neeman-book}. A further discussion of triangulated categories would lead us astray. Therefore, the author assumes here that the reader has some familiarity with triangulated categories, but would like to learn the connection to hereditary abelian model structures.

The easiest example of a triangulated category is probably the stable category St($\cat{F}$) of a Frobenius category $\cat{F}$. This construction is due to Happel~\cite{happel-triangulated}.

\begin{definition}\label{def-Frobenius}
A \textbf{Frobenius category} is an exact category in the sense of~\cite{quillen-algebraic K-theory}, that has enough projectives and enough injectives, but such that the projective objects coincide with the injective objects.
\end{definition}

Two common examples of Frobenius categories are $k[G]$-Mod, where $k$ is a field and $G$ is a finite group, and $\ch_{dw}$, the category of chain complexes with the degreewise split short exact sequences.  Applying the following construction to these two examples lead, respectively, to St($k[G]$) from Section~\ref{section-stable module cat},  and to $K(R)$, the usual chain homotopy category of complexes.

\begin{construction}\label{construction-happel}
We now review Happel's construction of the triangulated category St($\cat{F}$), the \textbf{stable category} of a Frobenius category $\cat{F}$. We let $\omega$ denote the class of projective-injective objects. The objects of St($\cat{F}$) are the same as those in $\cat{F}$. The morphisms are $\Hom_{\textnormal{St}(\cat{F})}(X,Y) = \Hom_{\cat{F}}(X,Y)/\sim$ where $f \sim g$ if and only if $g-f$ factors through an object of $\omega$. Now given $X \in \cat{F}$, the shift functor $X \mapsto \Sigma X$ is computed by taking any admissible short exact sequence $0 \xrightarrow{} X \xrightarrow{} W \xrightarrow{} \Sigma X \xrightarrow{} 0$ with $W \in \omega$. Any two choices of short exact sequences lead to canonically equivalent objects in St($\cat{F}$). We now define exact triangles as follows: Given any short exact sequence $0 \xrightarrow{} X \xrightarrow{f} Y \xrightarrow{} Z \xrightarrow{} 0$, we form the pushout diagram below.
$$\begin{CD}
@. 0 @. 0 \\
@. @VVV  @VVV \\
0 @>>> X @>f>> Y @>>> Z @>>> 0 \\
@. @VVV  @VVV @| \\
0 @>>> W @>>> C_f @>>> Z @>>> 0 \\
@. @VVV  @VVV \\
@. \Sigma X @= \Sigma X \\
@. @VVV  @VVV \\
@. 0 @. 0
\end{CD}$$
By definition, $X \xrightarrow{f} Y \xrightarrow{} C_f \xrightarrow{} \Sigma X$ is a \emph{standard exact triangle}, and we call $C_f$ the \emph{cone on $f$}. Then, again by definition, a sequence of morphisms $A \xrightarrow{} B \xrightarrow{} C \xrightarrow{} \Sigma A$ is an \emph{exact triangle} if it is isomorphic in St($\cat{F}$) to a standard exact triangle. This gives St($\cat{F}$) the structure of a triangulated category by~\cite[Theorem~2.6]{happel-triangulated}.
\end{construction}

The following proposition is the key reason why $\textnormal{Ho}(\class{M})$ turns out to be triangulated when $\class{M}$ is hereditary.

\begin{proposition}\label{prop-hocat-Froben}
Let $\class{M} = (\class{Q},\class{W},\class{R})$ be an hereditary abelian model structure. Then the full subcategory $\class{Q} \cap \class{R}$, along with the collection of short exact sequences with all three terms in $\class{Q} \cap \class{R}$, is a Frobenius category. The projective-injective objects are precisely those in the core $\omega = \class{Q}\cap\class{W}\cap\class{R}$.
\end{proposition}

\begin{proof}
This observation was first made in~\cite{gillespie-exact model structures}, but it is so easy that we will reprove it here to make it transparent. First, any strictly full subcategory $\cat{S}$ of an abelian category $\cat{A}$, in which $\cat{S}$ is closed under extensions, naturally inherits an exact structure from $\cat{A}$. It consists of the usual short exact sequences but with all three terms in $\cat{S}$. Since $\class{Q} \cap \class{R}$ is closed under extensions it inherits this exact structure. Second, we see that if $0 \xrightarrow{} X \xrightarrow{} Y \xrightarrow{} Z \xrightarrow{} 0$ is a short exact sequence with $X,Y,Z \in \class{Q} \cap \class{R}$, then it will split if either $X$ or $Z$ is in $\omega$, since this would imply $\Ext^1_{\cat{A}}(Z,X) = 0$. So everything in $\omega$ is both injective and projective with respect to the inherited exact structure on $\class{Q} \cap \class{R}$. On the other hand, say $I \in \class{Q} \cap \class{R}$ is injective with respect to the exact structure. Using the complete cotorsion pair $(\class{Q},\class{W} \cap \class{R})$, write a short exact sequence $0 \xrightarrow{} I \xrightarrow{} W \xrightarrow{} Q \xrightarrow{} 0$ with $W \in \class{W} \cap \class{R}$ and $Q \in \class{Q}$. Since the model structure is hereditary, $Q \in \class{Q} \cap \class{R}$. Also note that $W \in \omega$. So we have an admissible short exact sequence, which must split by the hypothesis on $I$. We infer $I \in \omega$, since $\omega$ is closed under direct summands. Hence all injective objects of $\class{Q} \cap \class{R}$ are in $\omega$. Performing the same argument as the above for an \emph{arbitrary} $X \in \class{Q} \cap \class{R}$ (in place of $I$) shows that $\class{Q} \cap \class{R}$ has enough injectives. Finally, a similar argument with the cotorsion pair $(\class{Q} \cap \class{W},\class{R})$ shows that all projectives are in $\omega$, and that $\class{Q} \cap \class{R}$ has enough projectives.
\end{proof}

The above proposition along with the Fundamental Theorem~\ref{them-fundamental}~(i) indicate that $\textnormal{Ho}(\class{M})$ is a triangulated category whenever $\class{M}$ is hereditary. Hanno Becker made this precise in an appendix of his PhD thesis~\cite{becker-thesis}. His work shows that Happel's Construction~\ref{construction-happel} above has a natural generalization as follows. For any $X \in \cat{A}$, the shift functor $X \mapsto \Sigma X$ may be computed by taking any short exact sequence $0 \xrightarrow{} X \xrightarrow{} W \xrightarrow{} \Sigma X \xrightarrow{} 0$ with $W \in \class{W}$.
Again, any two choices of short exact sequences lead to canonically equivalent objects in $\textnormal{Ho}(\cat{M})$. Then we can define standard exact triangles and exact triangles, analogous to the method from Happel's Construction~\ref{construction-happel}. This will give $\textnormal{Ho}(\class{M})$ the structure of a triangulated category. We get the following result.

\begin{theorem}\label{them-hocat-triangulated}
Let $\class{M} = (\class{Q},\class{W},\class{R})$ be an hereditary model structure on $\cat{A}$. Then the inclusion $(\class{Q} \cap \class{R})/\sim \ \hookrightarrow \textnormal{Ho}(\class{M}),$ induced by the inclusion of the Frobenius category $\class{Q} \cap \class{R}$ into $\cat{A}$, is a triangle equivalence. Its inverse is bifibrant replacement.
\end{theorem}

\begin{proof}
We already know that the inclusion of categories is an equivalence from the Fundamental Theorem~\ref{them-fundamental}~(i). The key point is to see why the inverse, bifibrant replacement, is triangle preserving. This follows from Becker's ``Resolution Lemma'', \cite[Lemma~1.4.4 and its dual]{becker}. Indeed applying that lemma will show that any pushout diagram of the form
$$\begin{CD}
@. 0 @. 0 \\
@. @VVV  @VVV \\
0 @>>> X @>f>> Y @>>> Z @>>> 0 \\
@. @VVV  @VVV @| \\
0 @>>> W @>>> C_f @>>> Z @>>> 0 \\
@. @VVV  @VVV \\
@. \Sigma X @= \Sigma X \\
@. @VVV  @VVV \\
@. 0 @. 0
\end{CD}$$
is sent to another pushout diagram of this form when taking cofibrant and fibrant replacements as in Lemma~\ref{lemma-cof-replacements}. It means that standard exact triangles (and hence all exact triangles) of $\textnormal{Ho}(\class{M})$, are sent to standard exact triangles (resp. exact triangles) when performing a bifibrant replacement.
\end{proof}

\section{Model structures for derived categories}\label{sec-chain-complexes}

In the last section we saw how hereditary cotorsion pairs are connected to the construction of triangulated categories. The most important example of a triangulated category is probably the derived category of a ring or scheme. We now describe how complete hereditary cotorsion pairs give rise to model structures for derived categories.

We start with a general lemma which has not explicitly appeared in the literature. It is useful because it shows that the formal homotopy relation coming from a model structure on chain complexes, will typically coincide with the usual notion of chain homotopic maps. Recall that two chain maps $f,g : X \xrightarrow{} Y$, between chain complexes $X$ and $Y$, are \emph{chain homotopic} if there exists a collection of maps $\{s_n : X_n \xrightarrow{} Y_{n+1}\}$, for which $g_n - f_n = d_{n+1}s_n +s_{n-1}d_n$.

\begin{lemma}\label{lemma-chain-homotopy}
Let $\cat{A}$ be an abelian category. Assume $(\class{Q}, \class{W}, \class{R})$ is a Hovey triple in $\cha{A}$, the category of chain complexes, and that its core $\class{Q}\cap \class{W}\cap \class{R}$ is closed under suspensions (that is, shifts). Then each $W \in \class{Q}\cap \class{W}\cap \class{R}$ is contractible. Moreover, if $\class{W}$ contains all contractible complexes, then we have $$\text{Ho}(\cha{A})(X,Y) \cong \Hom_{\cha{A}}(QX,RY)/\sim$$ where $\sim$ is the usual chain homotopy relation. In this case we also note that $$\text{Ho}(\cha{A})(X,Y) \cong K(\class{Q} \cap \class{R}),$$ where  $K(\class{Q} \cap \class{R})$ denotes the full subcategory of $K(R)$ consisting of all complexes in $\class{Q} \cap \class{R}$.
\end{lemma}

\begin{proof}
To say $W$ is contractible is to say that the identity map $1_W$ is null homotopic, or in other words, $1_W$ is chain homotopic to 0. Our assumptions imply that $\Ext^1_{\cha{A}}(\Sigma W,W) = 0$. By a ``well known'' fact, for example see Lemma~\cite[Lemma~2.1]{gillespie}, this implies \emph{all} maps $W \xrightarrow{} W$ are null homotopic. In particular, $1_W$ is null homotopic.

By the Fundamental Theorem~\ref{them-fundamental}~(iii) we know $$\text{Ho}(\cha{A})(X,Y) \cong \Hom_{\cha{A}}(QX,RY)/\sim$$ where $f \sim g$ if and only if $g - f$ factors through some $W \in \class{Q} \cap \class{W} \cap \class{R}$.  Since $W$ is contractible, this implies $g-f$ is null homotopic; for example, see~\cite[Corollary~3.5]{gillespie-N-complexes}. Conversely, say $g-f$ is null homotopic, or equivalently, that it factors through some contractible complex $C$. Write a short exact sequence $0 \xrightarrow{} R' \xrightarrow{} Q' \xrightarrow{} C \xrightarrow{} 0$ with $Q' \in \class{Q}$ and $R' \in \class{R} \cap \class{W}$. If $\class{W}$ contains all contractible complexes, then $C \in \class{W}$. Hence $Q' \in \class{Q} \cap \class{W}$. Now in the factorization $g-f : QX \xrightarrow{\alpha} C \xrightarrow{\beta} RY$ we have that $\alpha$ must lift over $Q' \in \class{Q} \cap \class{W}$, since $\Ext^1(QX,R') = 0$. This shows that $g-f$ factors through an object of $\class{Q} \cap \class{W}$ and hence $f$ is formally homotopic to $g$ by~\cite[Prop.~4.4]{gillespie-exact model structures}.

The last statement is then a restatement of the Fundamental Theorem~\ref{them-fundamental}~(i).
\end{proof}

\subsection{Models for the derived category}\label{sec-more-derived-models}
In addition to the projective and injective model structures on $\ch$, from Section~\ref{sec-derived-models}, there is an abundance of ways to model the derived category $\class{D}(R)$. The ideas go back to~\cite{gillespie,gillespie-degreewise-model-strucs}, where given a cotorsion pair $(\class{X},\class{Y})$ in $R$-Mod, various methods for lifting them to $\ch$ were given. In particular, based on the definitions of DG-injective and DG-projective complexes from Section~\ref{sec-derived-models} the following notation was introduced in~\cite[Def.~3.3]{gillespie}.

\begin{notation}\label{notation-complexes-pairs}
For a given cotorsion pair $(\class{X},\class{Y})$ in $R$-Mod we define the following classes of chain complexes in $\ch$.
\begin{enumerate}
\item  $\tilclass{X}$ denotes the class of all exact chain complexes $X$ with cycles $Z_nX \in \class{X}$.

\item  $\tilclass{Y}$ denotes the class of all exact chain complexes $Y$ with cycles $Z_nY \in \class{Y}$.

\item $\dgclass{X}$ denotes the class of all chain complexes $X$ with components $X_n \in \class{X}$ and such that all chain maps $X \xrightarrow{} Y$ are null homotopic whenever $Y \in \tilclass{Y}$.

\item $\dgclass{Y}$ denotes the class of all chain complexes $Y$ with components $Y_n \in \class{Y}$ and such that all chain maps $X \xrightarrow{} Y$ are null homotopic whenever $X \in \tilclass{X}$.
\end{enumerate}
Note that $\class{X}$ is closed under extensions and so for any $X \in \tilclass{X}$ we have each component $X_n \in \class{X}$ too. Similarly, all $Y_n \in \class{Y}$ whenever $Y \in \tilclass{Y}$.
\end{notation}

In particular, starting with the canonical projective cotorsion pair ($\class{P},\class{A})$ in $R$-Mod, the class $\tilclass{P}$ turns out to equal the class of categorically projective chain complexes. They are the split exact complexes with projective components. The class $\dgclass{P}$ is precisely the class of DG-projective complexes. The analogous facts also hold for lifting the canonical injective cotorsion pair $(\class{A},\class{I})$.

Now as noted in Section~\ref{sec-derived-models}, to model the derived category $\class{D}(R)$ one seeks an abelian model structure on $\ch$ with the exact chain complexes as the trivial objects. We let $\class{E}$ denote the (thick) class of all exact chain complexes. We ask what are the precise conditions one needs to impose on $(\class{X},\class{Y})$ to get a Hovey triple $\class{M}_{(\class{X},\class{Y})} = (\dgclass{X},\class{E},\dgclass{Y})$. It was shown in~\cite[Section~3]{gillespie} that we do have cotorsion pairs $(\dgclass{X},\tilclass{Y})$ and $(\tilclass{X},\dgclass{Y})$. Moreover, these cotorsion pairs are hereditary if and only if $(\class{X},\class{Y})$ is hereditary, which in turn is equivalent to each of  the desired equalities $\tilclass{X} = \dgclass{X} \cap \class{E}$ and $\tilclass{Y} = \class{E} \cap \dgclass{Y}$.
Thus $\class{M}_{(\class{X},\class{Y})} = (\dgclass{X},\class{E},\dgclass{Y})$ is a Hovey triple if and only if the two associated cotorsion pairs are complete. This last question has been approached using Theorem~\ref{them-cogen-by-a-set}, not just in $\ch$, but in the context of Grothendieck categories, in~\cite{gillespie,gillespie-sheaves,gillespie-quasi-coherent} and~\cite{stovicek-Hill}. But perhaps the most satisfying solution came in~\cite{GangYang-Liu models} and~\cite{Ding-question-gillespie}. In the first, it is proved directly that the cotorsion pairs $(\dgclass{X},\tilclass{Y})$ and $(\tilclass{X},\dgclass{Y})$ are complete whenever $(\class{X},\class{Y})$ is a complete hereditary cotorsion pair. This was then generalized to bicomplete abelian categories in~\cite{Ding-question-gillespie}, giving the following result.

\begin{theorem}[Models for Derived Categories]\label{them-derivedcat-models}
Let $\cat{A}$ be a bicomplete abelian category. Then $\class{M}_{(\class{X},\class{Y})} = (\dgclass{X},\class{E},\dgclass{Y})$ is a Hovey triple if and only if $(\class{X},\class{Y})$ is a complete hereditary cotorsion pair in $\cat{A}$. In this case $\class{M}_{(\class{X},\class{Y})}$ is also hereditary and $\textnormal{Ho}(\class{M}_{(\class{X},\class{Y})}) \cong \class{D}(\cat{A})$.
\end{theorem}

The original motive for Theorem~\ref{them-derivedcat-models} was for the construction of the \emph{flat model structure} on chain complexes of (quasi-coherent) sheaves~\cite{gillespie-quasi-coherent}.
We saw in Section~\ref{sec-derived-models} that the projective model structure formalizes the fact that $\Ext^n_R$ can be computed with projective resolutions while the injective model structure corresponds to its computation with injective resolutions. On the other hand, it is well known that $\Tor^R_n$, the derived tensor product, can be computed with projective resolutions but NOT with injectives. From the abelian model category point of view, this corresponds to the fact that the projective model is \emph{monoidal} while the injective model is not. Loosely speaking, a model structure on a category with a tensor product is monoidal if the model structure is compatible with that tensor product. In this case, the main consequence is that the tensor product descends to a derived tensor product on the homotopy category which can then be computed with cofibrant replacement; see~\cite[Chapter~4]{hovey-model-categories} and~\cite[Section~7]{hovey}. Now the category $\Qco(X)$, of quasi-coherent sheaves on a scheme $X$, rarely has enough projectives. So the projective model structure for $\class{D}(X)$ need not exist. Since $\Qco(X)$ is a Grothendieck category, the injective model structure does exist, but again it is not monoidal. This leaves a distressing gap in approaching the derived category from the model category perspective. To solve the problem one turns, not surprisingly, to flat sheaves. A famous result in relative homological algebra is that Enochs' flat cotorsion pair $(\class{F},\class{C})$ in $R$-Mod is complete (and hereditary). Applying Theorem~\ref{them-derivedcat-models} to it gives rise to a Hovey triple $(\dgclass{F},\class{E},\dgclass{C})$ in $\ch$. The model structure was shown to be monoidal in~\cite{gillespie} and it was extended to chain complexes of sheaves and quasi-coherent sheaves in~\cite{gillespie-sheaves,gillespie-quasi-coherent}. One benefit of having this monoidal model structure is that it follows at once that the triangulated structure on the homotopy category $\class{D}(X)$ is strongly compatible with the derived tensor product; see~\cite{may}.

Finally, we point out that there are even more model structures for $\class{D}(R)$ than just the ones coming from Theorem~\ref{them-derivedcat-models}. Given any cotorsion pair $(\class{X},\class{Y})$ in $R$-Mod, assumed to be cogenerated by a set $\class{S}$, it lifts to two ``degreewise'' model structures for the derived category; see~\cite{gillespie-degreewise-model-strucs}.  To illustrate, Baer's criterion implies that the categorical injective cotorsion pair $(\class{A},\class{I})$ is cogenerated by the set $\{R/I\}$ where $I$ ranges through all (left) ideals of $R$. It lifts to the \emph{degreewise injective cotorsion pair} on $\ch$. It is represented by the Hovey triple $(\leftperp{(\dwclass{I}\cap \class{E})},\class{E},\dwclass{I})$, where $\dwclass{I}$ denotes the class of all complexes $I$ with each $I_n$ injective; see~\cite[Section~4.1]{gillespie-degreewise-model-strucs}. Becker shows in~\cite{becker} that this model structure is the right Bousfield localization of the \emph{Inj model structure} of Proposition~\ref{prop-coderived} by the \emph{exact Inj model structure} of Proposition~\ref{prop-inj-stable}.

\section{Models for the coderived and contraderived categories} \label{sec-coderived-models}

Two particular localizations of $K(R)$, called the coderived and contraderived categories of a ring $R$, were introduced by Positselski in~\cite{positselski}. If $R$ has finite global dimension then they each coincide with the usual derived category $\class{D}(R)$. But for a general ring $R$, they are different and each contain $\class{D}(R)$. In this section we give brief constructions of these categories in terms of abelian model structures recently appearing in~\cite{positselski,bravo-thesis,bravo-gillespie-hovey,becker,becker-thesis,stovicek-purity,gillespie-mock projectives}. We will use the following notation.

\begin{notation}\label{notation-complexes}
Let $\class{C}$ denote a given class of $R$-modules, or objects in some other abelian category.
We set the following notation for the two classes of chain complexes below.
\begin{enumerate}
\item $\dwclass{C}$ denotes the class of all chain complexes $X$ with components $X_n \in \class{C}$.

\item $\exclass{C}$ denotes the class of all exact chain complexes $X$ with components $X_n \in \class{C}$.
\end{enumerate}
The \emph{dw} is meant to indicate that the complexes are ``degreewise'' in $\class{C}$, while the \emph{ex} is meant to emphasize that the complexes are ``exact'' and degreewise in $\class{C}$.
\end{notation}

\subsection{The coderived category}
The canonical injective cotorsion pair $(\cat{A},\class{I})$ in $R$-Mod can easily be shown to lift to a complete cotorsion pair $(\class{W}_{\textnormal{co}},\dwclass{I})$ in $\ch$. Indeed one lets $\class{S} = \{D^n(R/I)\}_{n \in \Z}$ as $I$ ranges through all (left) ideals of $R$. It cogenerates a complete cotorsion pair by Theorem~\ref{them-cogen-by-a-set} which turns out to be $(\class{W}_{\textnormal{co}},\dwclass{I})$; see~\cite[Prop.~4.4]{gillespie-degreewise-model-strucs}. As in Notation~\ref{notation-complexes}, $\dwclass{I}$ is the class of all degreewise injective complexes. The class $\class{W}_{\textnormal{co}}$ consists of all complexes $W$ such that $W \xrightarrow{} I$ is null homotopic whenever $I \in \dwclass{I}$. Following Positselski, such complexes are called \textbf{coacyclic}. All contractible complexes are coacyclic, and all coacyclic complexes are exact. But for general rings, an exact complex need not be coacyclic. Still, the class $\class{W}_{\textnormal{co}}$ is thick, and in analogy to the usual derived category, one obtains the coderived category by killing the coacyclic complexes. The cotorsion pair $(\class{W}_{\textnormal{co}}, \dwclass{I})$ is an \emph{injective cotorsion pair}, meaning it determines an abelian model structure on $\ch$ in which all objects are cofibrant. The following proposition sums all this up.

\begin{proposition}\label{prop-coderived}
The triple $\class{M}^{inj}_{\textnormal{co}} = (All, \class{W}_{\textnormal{co}}, \dwclass{I})$ is an hereditary abelian model structure on $\ch$. Its homotopy category, $\textnormal{Ho}(\class{M}^{inj}_{\textnormal{co}})$, is called the \textbf{coderived category} and it is equivalent to $K(Inj)$, the chain homotopy category of all complexes of injectives. That is, $\textnormal{Ho}(\class{M}^{inj}_{\textnormal{co}}) \cong K(Inj)$.
\end{proposition}

\begin{proof}
This model was constructed in~\cite[Cor.~4.4]{bravo-gillespie-hovey} where it was called the \textbf{Inj model structure} and independently in~\cite[Prop.~1.3.6]{becker} where it was called the \textbf{coderived model structure}. It generalizes to Grothendieck categories; see~\cite{stovicek-purity} and~\cite{gillespie-injective models}. The fact that $\textnormal{Ho}(\class{M}^{inj}_{\textnormal{co}}) \cong K(Inj)$ follows from Lemma~\ref{lemma-chain-homotopy}.
\end{proof}

Stovicek has found an alternate model structure for the coderived category and uses it to show that $K(Inj)$ is compactly generated whenever $R$ is coherent. This comes from~\cite{stovicek-purity}. To describe it, let $\class{A}$ be the class of \emph{absolutely pure}, or \emph{FP-injective} modules. By definition, such modules $A$ satisfy $\Ext^1_{R}(F,A) = 0$ for all finitely presented $F$.
By Theorem~\ref{them-cogen-by-a-set}, the set of (isomorphism representatives of) all finitely presented modules cogenerates a complete cotorsion pair $(\class{C},\class{A})$. Stovicek shows it to be hereditary if and only if $R$ is coherent.  It lifts to two complete hereditary cotorsion pairs $(\leftperp{\dwclass{A}},\dwclass{A})$ and $(\dgclass{C},\tilclass{A})$, by~\cite{gillespie,gillespie-degreewise-model-strucs}. It is not difficult to argue that these two cotorsion pairs have the same core. It is the class of contractible complexes with components in $\class{C} \cap \class{A}$. Thus Theorem~\ref{them-how to construct triples} immediately produces a model structure $\class{M} = (\dgclass{C}, \class{W}, \dwclass{A})$ on $\ch$. It is not clear from Theorem~\ref{them-how to construct triples} exactly which complexes are in $\class{W}$. But Stovicek proves that $\class{W} = \class{W}_{\textnormal{co}}$, yielding the following theorem.

\begin{theorem}[\cite{stovicek-purity}, Theorem 6.12]\label{them-stovicek-model}
Let $R$ be any coherent ring. Then  $\class{M}^{abs}_{\textnormal{co}} = (\dgclass{C}, \class{W}_{\textnormal{co}}, \dwclass{A})$ is an hereditary abelian model structure on $\ch$. Again, $\class{W}_{\textnormal{co}}$ is exactly the class of coacyclic complexes. Consequently, we have a triangle equivalence $\textnormal{Ho}(\class{M}^{abs}_{\textnormal{co}}) \cong K(Inj)$.
\end{theorem}

Critical to completing the proof of Theorem~\ref{them-stovicek-model} is showing that  $\class{W}_{\textnormal{co}} \cap \dwclass{A} = \tilclass{A}$. This can be viewed as the dual of an important result of Neeman that we will see shortly, just below. Stovicek's proof of this depends on a deeper understanding of pure exact complexes and utilizes a correspondence, due to Crawley-Boevey~\cite{crawley-boevey}, between pure exact complexes and categorically flat complexes. We refer the reader to~\cite[Prop.~6.11]{stovicek-purity} for details.

The main advantage to the model structure in Theorem~\ref{them-stovicek-model} is that it makes clear why $K(Inj)$ is compactly generated. The point is that the set of all spheres $S^n(F)$, where $F$ ranges through (a set of isomorphism representatives of) finitely presented modules, is a set of \emph{compact weak generators} for $\textnormal{Ho}(\class{M}^{abs}_{\textnormal{co}})$. It means that (i) Each functor $\Hom_{\textnormal{Ho}(\class{M}^{abs}_{\textnormal{co}})}(S^n(F),-)$ preserves coproducts, and, (ii) If $\Hom_{\textnormal{Ho}(\class{M}^{abs}_{\textnormal{co}})}(S^n(F),X) = 0$ for each $S^n(F)$, then $X \in \class{W}_{\textnormal{co}}$.

\begin{corollary}
Let $R$ be a coherent ring. Then $\textnormal{Ho}(\class{M}^{abs}_{\textnormal{co}})$ is compactly generated by the set $\{S^n(F)\}$,  as $F$ ranges through (a set of isomorphism representatives of) the finitely presented modules. Consequently, $K(Inj)$ is also compactly generated.
\end{corollary}

We point out that the above results hold in the general setting of a locally coherent Grothendieck category. Again, see~\cite{stovicek-purity}. The fact that $K(Inj)$ is compactly generated when $R$ is Noetherian has been known since~\cite{krause-stable derived cat of a Noetherian scheme}.

\subsection{The contraderived category}
On the other hand, the canonical projective cotorsion pair $(\class{P},\class{A})$ in $R$-Mod lifts to a complete cotorsion pair $(\dwclass{P},\class{W}_{\textnormal{ctr}})$ in $\ch$. Following Notation~\ref{notation-complexes}, $\dwclass{P}$ is the class of all degreewise projective complexes and the complexes in $\class{W}_{\textnormal{ctr}}$ have been called \textbf{contraacyclic}. The class  $\class{W}_{\textnormal{ctr}}$ is also thick, and killing the contraacyclic complexes produces the \textbf{contraderived category}.  This time, the cotorsion pair $(\dwclass{P},\class{W}_{\textnormal{ctr}})$ is a \emph{projective cotorsion pair}, giving rise to a model structure in which all complexes are fibrant, as follows.

\begin{proposition}\label{prop-contraderived}
The triple $\class{M}^{proj}_{\textnormal{ctr}} = (\dwclass{P}, \class{W}_{\textnormal{ctr}}, All)$ is an hereditary abelian model structure on $\ch$. Its homotopy category, $\textnormal{Ho}(\class{M}_{\textnormal{ctr}})$, is called the \textbf{contraderived category} and it is equivalent to $K(Proj)$, the chain homotopy category of all complexes of projectives. That is, $\textnormal{Ho}(\class{M}^{proj}_{\textnormal{ctr}}) \cong K(Proj)$.
\end{proposition}

\begin{proof}
Again, this model was constructed in~\cite[Cor.~6.4]{bravo-gillespie-hovey} where it was called the \textbf{Proj model structure} and independently in~\cite[Prop.~1.3.6]{becker} where it was called the \textbf{contraderived model structure}.
\end{proof}

Unlike the above model $\class{W}^{inj}_{\textnormal{co}}$, for $K(Inj)$, the projective model $\class{M}^{proj}_{\textnormal{ctr}}$ does not generalize to all Grothendieck categories. The problem again is that a general Grothendieck category need not have enough projectives.  However, the original inspiration for Grothendieck categories was a general framework for working with categories of (quasi-coherent) sheaves. Such categories typically have a generating set consisting of flat sheaves, or equivalently, they ``have enough flat sheaves''. With this in mind, Neeman showed in~\cite{neeman-flat} that $K(Proj)$ is equivalent to $\class{D}(Flat)$, the derived category of all flat modules. By definition, this turns out to be the Verdier quotient of $K(Flat)$, the chain homotopy category of all complexes with flat components, by $\tilclass{F}$.    As in Notation~\ref{notation-complexes}, $\tilclass{F}$ is the thick subcategory consisting of all exact complexes $F$ with each cycle $Z_nF$ flat. The complexes in $\tilclass{F}$ are categorically flat in $\ch$; they are direct limits of finitely generated projective complexes~\cite{enochs-GR-flat-98}. This construction of $\class{D}(Flat) = K(Flat)/\tilclass{F}$ extends to quasi-coherent sheaf categories. The main motivation for this is that $\class{D}(Flat)$ replaces $K(Proj)$ and plays a key role in an extension of the classical Grothendieck duality theory. The details of this are nicely explained in the Introduction to Daniel Murfet's PhD thesis~\cite{murfet-thesis}, and the full details carried out throughout the thesis.

From the abelian model category point of view, this all suggests that there ought to be an abelian model structure on $\ch$ whose homotopy category is equivalent to $\class{D}(Flat) = K(Flat)/\tilclass{F}$. This was carried out in~\cite{gillespie-mock projectives}.  The construction is dual to the construction of the model in Theorem~\ref{them-stovicek-model}. In detail, letting $(\class{F},\class{C})$ denote Enochs' flat cotorsion pair, $(\dwclass{F}, \rightperp{\dwclass{F}})$  and $(\tilclass{F}, \dgclass{C})$ are each complete hereditary cotorsion pairs. See~\cite{gillespie,gillespie-degreewise-model-strucs}. One argues that each has the same core - the contractible complexes with components in $\class{F} \cap \class{C}$. Theorem~\ref{them-how to construct triples} immediately gives a model structure $\class{M} = (\dwclass{F}, \class{W},\dgclass{C})$. But the key is a nontrivial result of Neeman from~\cite{neeman-flat}. In Notation~\ref{notation-complexes}, it states that $\dwclass{F} \cap \class{W}_{\textnormal{ctr}} = \tilclass{F}$, and this implies $\class{W} = \class{W}_{\textnormal{ctr}}$. It proves the following theorem.

\begin{theorem}[\cite{gillespie-mock projectives}, Corollary 4.1]\label{them-contra-model}
Let $R$ be any ring. Then there is an hereditary abelian model structure $\class{M}^{flat}_{\textnormal{ctr}} = (\dwclass{F},  \class{W}_{\textnormal{ctr}},\dgclass{C})$ where $\class{W}_{\textnormal{ctr}}$ is precisely the class of contraacyclic complexes. Consequently, $\textnormal{Ho}(\class{M}^{flat}_{\textnormal{ctr}}) \cong K(Proj) \cong \class{D}(Flat)$.
\end{theorem}

There are some notable things that follow from the model category approach to $\class{D}(Flat)$. First, the theorem lifts the Neeman-Murfet homotopy category $\class{D}(Flat)$, from a localization of chain complexes with flat components, to a localization of the entire chain complex category! Second, it provides a definition of contraacyclic complexes, even in the absence of projective objects. Indeed a generalization of Theorem~\ref{them-contra-model} holds universally, with the only assumption on the scheme $X$ being that the category of quasi-coherent sheaves has a flat generator. See~\cite{gillespie-mock projectives} for full details.

The author does not claim that the model category approach makes \emph{everything} more transparent. For example, these model structures make it no more obvious that $K(Proj)$ is a compactly generated category whenever $R$ is coherent. But this was shown to be true by J\o rgensen in~\cite{jorgensen-proj}, Neeman in~\cite{neeman-flat}, and for the sheaf case, by Murfet in~\cite{murfet-thesis}.

\section{Models for the stable derived categories} \label{sec-stable-derived-models}

We can view the category $K(Inj)$ as being obtained from the derived category $\class{D}(R)$ via an attachment, or recollement, with another category called the (injective) stable derived category. This line of thought was introduced by Krause in~\cite{krause-stable derived cat of a Noetherian scheme}.  Similarly, $K(Proj)$ is obtained from $\class{D}(R)$ by a recollement with the projective stable derived category. In this section we describe the stable derived categories using abelian model structures. We give a model category description of the recollements in Section~\ref{sec-recollements}.

\subsection{The injective stable derived category}
Again, we let $R$ be a general ring. Its \emph{injective stable derived category}, which we will denote $K_{ex}(Inj)$, is the full subcategory of $K(Inj)$ consisting of all exact complexes. It was introduced by Krause in~\cite{krause-stable derived cat of a Noetherian scheme} in the context of a locally noetherian Grothendieck category with compactly generated derived category. Certainly $R$-Mod is such a category whenever $R$ is a Noetherian ring, and Krause showed $K_{ex}(Inj)$ to be a compactly generated triangulated category in this case. Moreover, he shows that the subcategory of compact objects of $K_{ex}(Inj)$ is up to direct summands equivalent to the singularity category $\class{D}^b(R\textnormal{-mod})/\class{D}(R)^c$. By definition, this is the Verdier quotient of the bounded derived category of finitely generated $R$-modules by the subcategory of compact objects in $\class{D}(R)$, which are the so-called \emph{perfect} complexes. Further details on this geometric motivation can be found in~\cite{krause-stable derived cat of a Noetherian scheme,becker,stovicek-purity}. Our focus here is on describing how $K_{ex}(Inj)$ can be recovered as the homotopy category of some hereditary abelian model structures.

An injective model structure for $K_{ex}(Inj)$ was constructed by Daniel Bravo in~\cite{bravo-thesis} and independently by Hanno Becker in~\cite{becker}. It also appears in~\cite{bravo-gillespie-hovey}. It can be described quite simply by the single cotorsion pair $(\class{V}_{\textnormal{inj}},\exclass{I})$, where following Notation~\ref{notation-complexes}, $\exclass{I}$ is the class of all exact complexes with injective components. This cotorsion pair goes back to~\cite[Prop.~4.6/Sec.~4.1]{gillespie-degreewise-model-strucs}, where it is shown that $(\class{V}_{\textnormal{inj}},\exclass{I})$ is cogenerated by the set $\{S^n(R)\}\cup\{D^n(R/I)\}$ where $I$ ranges through all (left) ideals of $R$. In light of Theorem~\ref{them-cogen-by-a-set} this shows the cotorsion pair to be complete, however, it is not shown there that $\class{V}_{\textnormal{inj}}$ is thick and that $\class{V}_{\textnormal{inj}} \cap \exclass{I}$ coincides with the class of all categorically injective complexes. This is made clear in each of~\cite{bravo-thesis,becker,bravo-gillespie-hovey} and proves the following result.

\begin{proposition}\label{prop-inj-stable}
The triple  $\class{M}^{inj}_{\textnormal{stb}} = (All, \class{V}_{\textnormal{inj}}, \exclass{I})$ is an hereditary abelian model structure on $\ch$, called the \textbf{exact Inj model structure}. We call its homotopy category, $\textnormal{Ho}(\class{M}^{inj}_{\textnormal{stb}})$, the \textbf{injective stable derived category} as it is equivalent to $K_{ex}(Inj)$, the chain homotopy category of all exact complexes of injectives. That is, $\textnormal{Ho}(\class{M}^{inj}_{\textnormal{stb}}) \cong K_{ex}(Inj)$.
\end{proposition}

Using this model structure one can deduce that the set of all spheres $S^n(F)$, where $F$ ranges through (a set of isomorphism representatives of) finitely generated modules, is a set of compact weak generators for $\textnormal{Ho}(\class{M}^{inj}_{\textnormal{stb}})$ in the case $R$ is Noetherian. However, for coherent rings that are not Noetherian we need an alternate model structure for $K_{ex}(Inj)$, one based on the absolutely pure modules, to reach this conclusion. With this in mind, recall from Section~\ref{sec-coderived-models} that we have the complete cotorsion pair $(\class{C},\class{A})$ where $\class{A}$ are the absolutely pure modules. It is hereditary if and only if $R$ is coherent. In this case, we obtain an hereditary complete cotorsion pair $(\leftperp{\exclass{A}},\exclass{A})$, by~\cite{gillespie-degreewise-model-strucs}. Applying Theorem~\ref{them-how to construct triples} to  $(\leftperp{\exclass{A}},\exclass{A})$ and $(\dgclass{C}, \tilclass{A})$ leads to the following model structure.

\begin{theorem}\label{them-exact-abs-model}
Let $R$ be a coherent ring. Then $\class{M}^{abs}_{\textnormal{stb}} = (\dgclass{C}, \class{V}_{\textnormal{inj}}, \exclass{A})$ is an hereditary abelian model structure on $\ch$. The class $\class{V}_{\textnormal{inj}}$ is exactly the same class as in Proposition~\ref{prop-inj-stable}. Consequently, we have a triangle equivalence $\textnormal{Ho}(\class{M}^{abs}_{\textnormal{stb}}) \cong K_{ex}(Inj)$.
\end{theorem}

Using Theorem~\ref{them-exact-abs-model} one can show that $K_{ex}(Inj)$ is compactly generated whenever $R$ is coherent. Again, it is the set of all spheres $S^n(F)$, where $F$ ranges through (a set of isomorphism representatives of) finitely presented modules, that provides a set of compact weak generators for $\textnormal{Ho}(\class{M}^{inj}_{\textnormal{stb}})$.  

\subsection{The projective stable derived category}

Dual to the injective stable derived category, is the projective stable derived category $K_{ex}(Proj)$. It is the full subcategory of $K(Proj)$ consisting of all exact complexes of projectives. Letting $\exclass{P}$ denote the class of all exact complexes of projectives, it is shown in both~\cite{bravo-gillespie-hovey} and~\cite{becker} that we have an hereditary cotorsion pair $(\exclass{P}, \class{V}_{\textnormal{prj}})$. The class $\class{V}_{\textnormal{prj}}$ is thick and $\exclass{P} \cap \class{V}_{\textnormal{prj}}$ coincides with the class of all categorically projective complexes. This proves the following result.

\begin{proposition}\label{prop-proj-stable}
The triple $\class{M}^{proj}_{\textnormal{stb}} = (\exclass{P}, \class{V}_{\textnormal{prj}}, All)$  is an hereditary abelian model structure on $\ch$, we call the \textbf{exact Proj model structure}. We call its homotopy category, $\textnormal{Ho}(\class{M}^{proj}_{\textnormal{stb}})$, the \textbf{projective stable derived category} and it is equivalent to $K_{ex}(Proj)$, the chain homotopy category of all exact complexes of projectives. That is, $\textnormal{Ho}(\class{M}^{proj}_{\textnormal{stb}}) \cong K_{ex}(Proj)$.
\end{proposition}

The geometric significance of $K_{ex}(Proj)$ is explained quite well in the Introduction to Murfet's thesis~\cite{murfet-thesis}. Recall from Section~\ref{sec-coderived-models} that $K(Proj) \cong \class{D}(Flat) = K(Flat)/\tilclass{F}$. Murfet showed that the full subcategory of $\class{D}(Flat)$ consisting of the exact complexes plays the roll of $K_{ex}(Proj)$ for non-affine schemes. To obtain this category as the homotopy category of an abelian model structure we again let $(\class{F},\class{C})$ denote Enochs' flat cotorsion pair. Then by~\cite{gillespie-degreewise-model-strucs} it lifts to two hereditary complete cotorsion pairs $(\exclass{F}, \rightperp{\exclass{F}})$  and $(\tilclass{F}, \dgclass{C})$.  Theorem~\ref{them-how to construct triples} applies yet again and immediately gives a model structure $\class{M} = (\exclass{F}, \class{V},\dgclass{C})$. One argues again that $\class{V} = \class{V}_{\textnormal{prj}}$ in the affine case, and this proves the following theorem.

\begin{theorem}\label{them-stable-flat-model}
Let $R$ be any ring. Then there is an hereditary abelian model structure $\class{M}^{flat}_{\textnormal{stb}} = (\exclass{F},  \class{V}_{\textnormal{prj}},\dgclass{C})$ where $\class{V}_{\textnormal{prj}}$ is the same class of trivial objects as in Proposition~\ref{prop-proj-stable}. Consequently, $\textnormal{Ho}(\class{M}^{flat}_{\textnormal{stb}}) \cong K_{ex}(Proj)$.
\end{theorem}

Again, the point is that Theorem~\ref{them-stable-flat-model} extends from the category of chain complexes of $R$-modules to the category of chain complexes of quasi-coherent sheaves. We only need  to assume that the category $\Qco(X)$, of quasi-coherent sheaves on a scheme $X$, has a flat generator. We refer the reader to~\cite{gillespie-mock projectives} for full details.

\section{Abelian model categories and recollements}\label{sec-recollements}

A recollement is an ``attachment'' or ``gluing'' of two triangulated categories. They were introduced in~\cite{BBD-perverse sheaves} and are now part of the lore of algebraic geometry and the theory of triangulated categories. Recall from Section~\ref{sec-triangulated} that the homotopy category of an hereditary abelian model structure is indeed a triangulated category.  In~\cite{becker}, Becker indicates how to obtain recollements via abelian model categories. This line of thought was pursued in much more detail by the current author. In this section we describe a main theorem which outputs a recollement from three interrelated hereditary Hovey triples. We then use it to recover the injective recollement of~\cite{krause-stable derived cat of a Noetherian scheme,becker,stovicek-purity} and the dual projective recollement of~\cite{jorgensen-proj,neeman-flat,murfet-thesis}. Loosely speaking, the projective recollement shows that the contraderived category is obtained by ``gluing'' the usual derived category to the projective stable derived category. The model category approach described here has been distilled  from~\cite{becker,gillespie-recollements,gillespie-recoll2,gillespie-mock projectives}.

\subsection{Definition of recollement}
We start with the definition of a recollement. Here we follow the definition from~\cite{krause-stable derived cat of a Noetherian scheme} which is based on localization and colocalization sequences.

\begin{definition}\label{def-localization sequence}
Let $\class{T}' \xrightarrow{F} \class{T} \xrightarrow{G} \class{T}''$ be a sequence of exact functors between triangulated categories. We say it is a \textbf{localization sequence} when there exists right adjoints $F_{\rho}$ and $G_{\rho}$ giving a diagram of functors as below with the listed properties.
$$\begin{tikzcd}
\class{T}'
\rar[to-,
to path={
([yshift=0.5ex]\tikztotarget.west) --
([yshift=0.5ex]\tikztostart.east) \tikztonodes}][swap]{F}
\rar[to-,
to path={
([yshift=-0.5ex]\tikztostart.east) --
([yshift=-0.5ex]\tikztotarget.west) \tikztonodes}][swap]{F_{\rho}}
& \class{T}
\rar[to-,
to path={
([yshift=0.5ex]\tikztotarget.west) --
([yshift=0.5ex]\tikztostart.east) \tikztonodes}][swap]{G}
\rar[to-,
to path={
([yshift=-0.5ex]\tikztostart.east) --
([yshift=-0.5ex]\tikztotarget.west) \tikztonodes}][swap]{G_{\rho}}
& \class{T}'' \\
\end{tikzcd}$$
\begin{enumerate}
\item The right adjoint $F_{\rho}$ of $F$ satisfies $F_{\rho} \circ F \cong \text{id}_{\class{T}'}$.
\item The right adjoint $G_{\rho}$ of $G$ satisfies $G \circ G_{\rho} \cong \text{id}_{\class{T}''}$.
\item For any object $X \in \class{T}$, we have $GX = 0$ iff $X \cong FX'$ for some $X' \in \class{T}'$.
\end{enumerate}
A \textbf{colocalization sequence} is the dual. That is, there must exist left adjoints $F_{\lambda}$ and $G_{\lambda}$ with the analogous properties.
\end{definition}

One may think of a localization sequence as a sequence of left adjoints which ``splits'' at the level of triangulated categories. See~\cite[Section~3]{krause-localization theory for triangulated categories} for the first properties of localization sequences which reflect this statement. Similarly, a colocalization sequence is a sequence of right adjoints with this property. It is true that if  $\class{T}' \xrightarrow{F} \class{T} \xrightarrow{G} \class{T}''$ is a localization sequence then  $\class{T}'' \xrightarrow{G_{\rho}} \class{T} \xrightarrow{F_{\rho}} \class{T}'$ is a colocalization sequence and if  $\class{T}' \xrightarrow{F} \class{T} \xrightarrow{G} \class{T}''$ is a colocalization sequence then  $\class{T}'' \xrightarrow{G_{\lambda}} \class{T} \xrightarrow{F_{\lambda}} \class{T}'$ is a localization sequence. This brings us to the definition of a recollement where the sequence of functors  $\class{T}' \xrightarrow{F} \class{T} \xrightarrow{G} \class{T}''$ is a localization sequence ``glued'' to a colocalization sequence.

\begin{definition}\label{def-recollement}
Let $\class{T}' \xrightarrow{F} \class{T} \xrightarrow{G} \class{T}''$ be a sequence of exact functors between triangulated categories. We say $\class{T}' \xrightarrow{F} \class{T} \xrightarrow{G} \class{T}''$ induces a \textbf{recollement} if it is both a localization sequence and a colocalization sequence as shown in the picture.
\[
\xy
(-20,0)*+{\class{T}'};
(0,0)*+{\class{T}};
(20,0)*+{\class{T}''};
{(-18,0) \ar^{F} (-2,0)};
{(-2,0) \ar@/^1pc/@<0.5em>^{F_{\rho}} (-18,0)};
{(-2,0) \ar@/_1pc/@<-0.5em>_{F_{\lambda}} (-18,0)};
{(2,0) \ar^{G} (18,0)};
{(18,0) \ar@/^1pc/@<0.5em>^{G_{\rho}} (2,0)};
{(18,0) \ar@/_1pc/@<-0.5em>_{G_{\lambda}} (2,0)};
\endxy
\]
\end{definition}

\subsection{Recollements from abelian model structures}
We now let $\cat{A}$ be any abelian category and assume $\class{M} = (\class{Q},\class{W},\class{R})$ is a Hovey triple. To describe the main theorem we will need to introduce some notation. We will denote the two associated cotorsion pairs above by $(\tilclass{Q},\class{R})$ and $(\class{Q},\tilclass{R})$. Note that there are four approximation sequences associated to $\class{M}$. In particular, using enough injectives of $(\tilclass{Q},\class{R})$ corresponds to the fibrant replacement functor denoted by $R$; see Lemma~\ref{lemma-cof-replacements}. On the other hand, using enough projectives of $(\class{Q},\tilclass{R})$ corresponds to the cofibrant replacement functor which we denote by $Q$. However, by using enough projectives of $(\tilclass{Q},\class{R})$ we also get a functor which we denote by $\widetilde{Q}$ and by using enough injectives of $(\class{Q},\tilclass{R})$ we get a functor we denote by $\widetilde{R}$. When we encounter multiple abelian model structures we use subscripts such as $\class{M}_1,\class{M}_2,\class{M}_3$ and denote these associated functors with notations such as $R_3$, $\widetilde{Q}_1$, $\widetilde{R}_2$ etc.

Finally, for an abelian model structure $\class{M} = (\class{Q},\class{W},\class{R})$ on $\cat{A}$, we let $\gamma : \cat{A} \xrightarrow{} \textnormal{Ho}(\cat{M})$ denote the canonical localization functor of Theorem~\ref{them-fundamental}. It takes weak equivalences to isomorphisms and is universally initial with respect to functors with this property. Using this we obtain the following lemma which is the last thing needed to state the main result.

\begin{lemma}\label{lemma-quotient map}
Let $\cat{A}$ be any abelian category. Assume $\class{M} = (\class{Q},\class{W},\class{R})$ and $\class{M}' = (\class{Q},\class{W}',\class{R}')$ are Hovey triples with the same cofibrant objects. If $\class{R}' \subseteq \class{R}$, then we have a canonical \textbf{quotient functor} $\bar{\gamma} : \textnormal{Ho}(\cat{M}) \xrightarrow{} \textnormal{Ho}(\cat{M}')$ for which $\gamma' = \bar{\gamma} \circ \gamma$.
\end{lemma}

\begin{proof}
First we will show that $\class{W} \subseteq \class{W}'$. Indeed since $\class{R}' \subseteq \class{R}$, we have a containment of trivially cofibrant objects $\class{Q} \cap \class{W} \subseteq \class{Q} \cap \class{W}'$.  Similarly, since the two model structures share the same cofibrant objects we have equality of trivially fibrant objects $\class{W} \cap \class{R} = \class{W}' \cap \class{R}'$. Now let $W \in \class{W}$. We need to see why $W \in \class{W}'$. We write a short exact sequence $0 \xrightarrow{} W' \xrightarrow{} C \xrightarrow{} W \xrightarrow{} 0$ with $C \in \class{Q}$ and $W' \in \class{W}' \cap \class{R}'$. Then $W'$ must also be in $\class{W}$, and since $\class{W}$ is thick, we see that $C \in \class{Q} \cap \class{W} \subseteq \class{Q} \cap \class{W}'$. Since $\class{W}'$ is also thick we get $W \in \class{W}'$.

Now by Proposition~\ref{prop-characterization of weak equivs}, a map is a weak equivalence if and only if it factors as a monomorphism with trivial cokernel followed by an epimorphism with trivial kernel.
So since $\class{W} \subseteq \class{W}'$, the localization functor $\gamma' : \cat{A} \xrightarrow{} \textnormal{Ho}(\cat{M}')$ sends weak equivalences in $\class{M}$ to isomorphisms. So the universal property of $\gamma$ guarantees the unique functor  $\bar{\gamma} : \textnormal{Ho}(\cat{M}) \xrightarrow{} \textnormal{Ho}(\cat{M}')$ for which $\gamma' = \bar{\gamma} \circ \gamma$.

\end{proof}

We now state the main result.

\begin{theorem}[Right Recollement Theorem]\label{them-right recollement}
Let $\cat{A}$ be an abelian category with three hereditary model structures, as below, whose cores all coincide: $$\class{M}_1 = (\class{Q}, \class{W}_1, \class{R}_1) , \ \ \ \class{M}_2 = (\class{Q}, \class{W}_2, \class{R}_2) , \ \ \ \class{M}_3 = (\class{Q}, \class{W}_3, \class{R}_3).$$ If $\class{W}_3 \cap \class{R}_1 = \class{R}_2$ and $\class{R}_3 \subseteq  \class{R}_1$, then the sequence
$\textnormal{Ho}(\class{M}_2) \xrightarrow{R_2} \textnormal{Ho}(\class{M}_1) \xrightarrow{\bar{\gamma}} \textnormal{Ho}(\class{M}_3)$ induces a recollement:
\[
\begin{tikzpicture}[node distance=3.5cm]
\node (A) {$\textnormal{Ho}(\class{M}_2)$};
\node (B) [right of=A] {$\textnormal{Ho}(\class{M}_1)$};
\node (C) [right of=B] {$\textnormal{Ho}(\class{M}_3)$};
\draw[<-,bend left=40] (A.20) to node[above]{$Q_1$} (B.160);
\draw[->] (A) to node[above]{\small $R_2$} (B);
\draw[<-,bend right=40] (A.340) to node [below]{$\widetilde{Q}_3 \circ R_1$} (B.200);

\draw[<-,bend left] (B.20) to node[above]{\small $\widetilde{Q}_2$} (C.160);
\draw[->] (B) to node[above]{$\bar{\gamma}$} (C);
\draw[<-,bend right] (B.340) to node [below]{\small $R_3$} (C.200);
\end{tikzpicture}\]
Here, the functor $\bar{\gamma}$ is the quotient functor of Lemma~\ref{lemma-quotient map}.
\end{theorem}

We refer the reader to~\cite[Section~3]{gillespie-mock projectives} for  a detailed proof of Theorem~\ref{them-right recollement}. The main idea is that the containment $\class{R}_2 \subseteq \class{R}_1$ produces the colocalization sequence on the top of the diagram, while the containment $\class{R}_3 \subseteq \class{R}_1$ produces the localization sequence on the bottom of the diagram. The equality $\class{W}_3 \cap \class{R}_1 = \class{R}_2$ provides the ``glue''. The proof is entirely (model) categorical. So there is an equally useful dual which we also state.

\begin{theorem}[Left Recollement Theorem]\label{them-left recollement}
Let $\cat{A}$ be an abelian category with three hereditary model structures, as below, whose cores all coincide: $$\class{M}_1 = (\class{Q}_1, \class{W}_1, \class{R}) , \ \ \ \class{M}_2 = (\class{Q}_2, \class{W}_2, \class{R}) , \ \ \ \class{M}_3 = (\class{Q}_3, \class{W}_3, \class{R}).$$ If $\class{W}_3 \cap \class{Q}_1 = \class{Q}_2$ and $\class{Q}_3 \subseteq  \class{Q}_1$, then the sequence
$\textnormal{Ho}(\class{M}_2) \xrightarrow{Q_2} \textnormal{Ho}(\class{M}_1) \xrightarrow{\bar{\gamma}} \textnormal{Ho}(\class{M}_3)$ induces a recollement:
\[
\begin{tikzpicture}[node distance=3.5cm]
\node (A) {$\textnormal{Ho}(\class{M}_2)$};
\node (B) [right of=A] {$\textnormal{Ho}(\class{M}_1)$};
\node (C) [right of=B] {$\textnormal{Ho}(\class{M}_3)$};
\draw[<-,bend left=40] (A.20) to node[above]{$\widetilde{R}_3 \circ Q_1$} (B.160);
\draw[->] (A) to node[above]{\small $Q_2$} (B);
\draw[<-,bend right=40] (A.340) to node [below]{$R_1$} (B.200);

\draw[<-,bend left] (B.20) to node[above]{\small $Q_3$} (C.160);
\draw[->] (B) to node[above]{$\bar{\gamma}$} (C);
\draw[<-,bend right] (B.340) to node [below]{\small $\widetilde{R}_2$} (C.200);
\end{tikzpicture}\]
Here, the functor $\bar{\gamma}$ is the quotient functor of (the dual of) Lemma~\ref{lemma-quotient map}.
\end{theorem}

\subsection{The recollements of Krause and Neeman-Murfet}

As applications of Theorems~\ref{them-right recollement} and~\ref{them-left recollement} we give model category interpretations of the recollements of Krause, from~\cite{krause-stable derived cat of a Noetherian scheme}, and the duals of Neeman and Murfet, from~\cite{neeman-flat,murfet-thesis}.

\begin{corollary}[Krause's recollement: \cite{stovicek-purity}, Theorem~7.7]\label{cor-krause-abs-recollement}
Assume $R$ is a coherent ring. Recall we have the following three model structures coming, respectively, from Theorem~\ref{them-stovicek-model}, Theorem~\ref{them-exact-abs-model}, and an application of Theorem~\ref{them-derivedcat-models}:
$$\class{M}^{abs}_{\textnormal{co}} = (\dgclass{C}, \class{W}_{\textnormal{co}}, \dwclass{A}) , \ \ \ \   \class{M}^{abs}_{\textnormal{stb}} = (\dgclass{C}, \class{V}_{\textnormal{inj}}, \exclass{A}) , \ \ \ \   \class{M}^{abs}_{\textnormal{der}} = (\dgclass{C}, \class{E}, \dgclass{A}).$$ Then they induce a recollement of compactly generated triangulated categories:
\[
\xy
(-30,0)*+{\textnormal{Ho}(\class{M}^{abs}_{\textnormal{stb}})};
(0,0)*+{\textnormal{Ho}(\class{M}^{abs}_{\textnormal{co}})};
(30,0)*+{\textnormal{Ho}(\class{M}^{abs}_{\textnormal{der}})};
{(-19,0) \ar (-10,0)};
{(-10,0) \ar@<0.5em> (-19,0)};
{(-10,0) \ar@<-0.5em> (-19,0)};
{(10,0) \ar (19,0)};
{(19,0) \ar@<0.5em> (10,0)};
{(19,0) \ar@<-0.5em> (10,0)};
\endxy
.\]
\end{corollary}

\begin{proof}
It follows immediately from Theorem~\ref{them-right recollement} since clearly $\class{E} \cap \dwclass{A} = \exclass{A}$.
\end{proof}

The above recollement comes from ``lifting'' the absolutely pure cotorsion pair $(\class{C},\class{A})$, which we recall is hereditary if and only if $R$ is coherent. For non-coherent rings we can even substitute for the absolutely pure modules the so called absolutely clean modules of Section~\ref{sec-stable-modules}; we refer the interested reader to~\cite{gillespie-injective models}. We should also note that the model structures and recollements of Corollary~\ref{cor-krause-abs-recollement} hold for quite general Grothendieck categories. Again, see~\cite{stovicek-purity,gillespie-injective models}.

On the other hand, we have the recollement of Neeman-Murfet. It relates the usual derived category to the projective stable derived category and the contraderived category. With the proper model structures already in place, we recover it from another easy application of the Left Recollement Theorem~\ref{them-left recollement}.

\begin{corollary}[Neeman-Murfet recollement: \cite{neeman-flat,murfet-thesis}]\label{cor-murfet-flat-recollement}
Let $R$ be any ring, or more generally, any scheme $X$ admitting a set of flat generators for $\Qco(X)$, the category of quasi-coherent sheaves on $X$. Recall we have the following three model structures coming, respectively, from Theorem~\ref{them-contra-model}, Theorem~\ref{them-stable-flat-model}, and an application of Theorem~\ref{them-derivedcat-models}:
$$\class{M}^{flat}_{\textnormal{ctr}} = (\dwclass{F},  \class{W}_{\textnormal{ctr}},\dgclass{C}) , \ \ \ \   \class{M}^{flat}_{\textnormal{stb}} = (\exclass{F},  \class{V}_{\textnormal{prj}},\dgclass{C}) , \ \ \ \   \class{M}^{flat}_{\textnormal{der}} = (\dgclass{F}, \class{E}, \dgclass{C}).$$ Then they induce a recollement of triangulated categories:
\[
\xy
(-30,0)*+{\textnormal{Ho}(\class{M}^{flat}_{\textnormal{stb}})};
(0,0)*+{\textnormal{Ho}(\class{M}^{flat}_{\textnormal{ctr}})};
(30,0)*+{\textnormal{Ho}(\class{M}^{flat}_{\textnormal{der}})};
{(-19,0) \ar (-10,0)};
{(-10,0) \ar@<0.5em> (-19,0)};
{(-10,0) \ar@<-0.5em> (-19,0)};
{(10,0) \ar (19,0)};
{(19,0) \ar@<0.5em> (10,0)};
{(19,0) \ar@<-0.5em> (10,0)};
\endxy
.\]
\end{corollary}

When $R$ is a coherent ring, the three homotopy categories are compactly generated as shown by J\o rgensen~\cite{jorgensen-proj} and Neeman~\cite{neeman-flat}, and for the sheaf case, by Murfet in~\cite{murfet-thesis}. But again, unlike the injective case, there doesn't seem to be an easy model category interpretation of this.

\section{Stable module categories}\label{sec-stable-modules}

We described in Section~\ref{section-stable module cat} how St($R$), the stable category of a quasi-Frobenius ring $R$, has a generalization to Gorenstein rings. Indeed Theorem~\ref{them-Goren-models} says that the stable module category of a Gorenstein ring is the homotopy category of all Gorenstein injective modules, or equivalently, all Gorenstein projective modules. This is formalized by the existence of the Hovey triples $(All, \class{W},\class{GI})$ and $(\class{GP},\class{W}, All)$ where $\class{W}$ is the class of all modules of finite projective (equivalently, finite injective) dimension. For non Gorenstein rings, one can't expect to always have a class such as $\class{W}$, ``balancing'' the Gorenstein injectives and the Gorenstein projectives in two Hovey triples. But we end this survey by discussing how a refinement of the traditional Gorenstein injective and Gorenstein projective modules leads to two abelian model structures on $R$-Mod for a general ring $R$. That is, we will describe Hovey triples $\mathcal{S}^{inj} = (All, \class{W}_{\textnormal{inj}},\class{GI})$ and $\mathcal{S}^{proj} = (\class{GP},\class{W}_{\textnormal{proj}}, All)$ in $R$-Mod with the property that $\class{W}_{\textnormal{inj}} = \class{W} = \class{W}_{\textnormal{proj}}$ when $R$ is Gorenstein.
We think of $\textnormal{Ho}(\mathcal{S}^{inj})$ as the \emph{injective stable module category} of $R$ and $\textnormal{Ho}(\mathcal{S}^{proj})$ as the \emph{projective stable module category} of $R$. This is a summary of the paper~\cite{bravo-gillespie-hovey} which grew out of the thesis of Daniel Bravo~\cite{bravo-thesis}.


As we will see, the idea is based on the observation that the traditional notion of Gorenstein injective modules only seems well suited for Noetherian rings. But in a classic paper, \cite{stenstrom-fp}, Stenstr{\"o}m described how many results in homological algebra can be extended from Noetherian rings to coherent rings by replacing finitely generated modules with finitely presented modules. In the process, the usual injective modules get replaced by the absolutely pure modules (also called FP-injective modules). Based on this, Ding and Mao introduced a refinement of Gorenstein injective modules for coherent rings in~\cite{ding and mao 08}. As was suggested in~\cite{gillespie-Ding-Chen rings}, these modules are now often called \emph{Ding injective}, and they turn out to give a quite satisfactory extension of Gorenstein injective modules to coherent rings. In particular, the Ding injectives are the fibrant objects of an injective abelian model structure on $R$-Mod whenever $R$ is coherent, yet they coincide with the usual Gorenstein injectives when $R$ is Noetherian. But in~\cite{bravo-gillespie-hovey}, the philosophy of Stenstr{\"o}m is taken all the way to general rings. When relaxing from coherent rings to general rings, it is noticed that the finitely presented modules should be replaced with the modules of type $FP_{\infty}$. The process results in a new class of modules, the \emph{absolutely clean} modules replacing by the absolutely pure modules. This in turn leads to the definition of \emph{Gorenstein AC-injective} modules as a further refinement of the Ding injective modules. We will explain this further throughout this section, along with the  the projective analog.

\subsection{The absolutely clean modules}\label{sec-absclean}

A ring $R$ is (left) Noetherian if and only if every finitely generated (left) $R$-module $F$ has a projective resolution $\cdots \rightarrow P_2 \rightarrow P_1 \rightarrow P_0 \rightarrow F \rightarrow 0$ with each $P_i$ a finitely generated projective. The existence of such resolutions is the reason why injective modules over Noetherian rings enjoy so many nice homological properties. As an example, we recall the reason why injective modules over Noetherian rings are closed under direct sums. Indeed, by Baer's criterion, a module $A$ is injective if and only if $\Ext^1_R(F,A) = 0$ for each finitely generated $F$. So given a collection $\{A_{\alpha}\}$ of injective modules $A_{\alpha}$, and letting $P_* \rightarrow F \rightarrow 0$ denote a projective resolution of $F$ by finitely generated projectives, we compute:
$$\text{Ext}^1_R(F, \bigoplus A_{\alpha}) = H^1[\text{Hom}_R(P_*, \bigoplus A_{\alpha})] \cong H^1[\bigoplus \text{Hom}_R(P_*, A_{\alpha})]$$
$$\cong \bigoplus H^1[\text{Hom}_R(P_*, A_{\alpha})] \cong \bigoplus \text{Ext}^1_R(F, A_{\alpha}) = 0.$$
The computation is only valid since $\Hom_R(P_*,-)$ preserves direct sums when each $P_i$ is finitely generated. Now if $R$ is just (left) coherent it is precisely the finitely presented modules $F$ which enjoy the property of having projective resolutions $P_* \rightarrow F \rightarrow 0$ with each $P_i$ finitely generated. Thus the above computation holds whenever $F$ is finitely presented and $\{A_{\alpha}\}$ is a collection of absolutely pure (that is, FP-injective) modules. In other words, the same computation shows that absolutely pure modules are closed under direct sums whenever $R$ is coherent. This indicates Stenstr{\"o}m's philosophy from~\cite{stenstrom-fp}: As a ring is relaxed from Noetherian to coherent, the appropriate notion of ``finite'' module should sharpen from the finitely generated modules to the finitely presented modules.

This motivates some definitions. For a general ring $R$, we say an $R$-module $F$ is of \textbf{type $\boldsymbol{FP_{\infty}}$} if it admits a projective resolution by finitely generated projectives. So for Noetherian rings these are the finitely generated modules and for coherent rings these are the finitely presented modules. A first study of modules of type $FP_{\infty}$ appeared in~\cite{bieri}. Among other things, he shows that the modules of type $FP_{\infty}$ always form a thick subcategory. That is, for any ring $R$, they are closed under direct summands and whenever two out of three terms in a short exact sequence are of type $FP_{\infty}$, then so is the third.  Further homological properties of such modules were studied by Livia Hummel in~\cite{miller-livia}. They have also appeared in group representation theory; see~\cite{benson-infinite} and~\cite{kropholler}. Thinking of modules of type $FP_{\infty}$ as the correct notion of ``finite'' modules over general rings, we are then led to make some more natural  definitions.

\begin{definition}\label{def-abs-clean}
We call a short exact sequence $0 \rightarrow A \rightarrow B \rightarrow C \rightarrow 0$ \textbf{clean} if it remains exact after applying $\text{Hom}_R(F,-)$ for any $F$ of type $FP_{\infty}$. We say an $R$-module $A$ is \textbf{absolutely clean}, or $FP_{\infty}$-injective, if
$\Ext^1_R(F,A) = 0$ for each module $F$ of type $FP_{\infty}$.
\end{definition}

Using standard arguments with $\Ext$, it is easy to see that a module $A$ is $FP_{\infty}$-injective if and only if every short exact sequence $0 \rightarrow A \rightarrow B \rightarrow C \rightarrow 0$ is clean. Hence the name ``absolutely clean'', as this is completely analogous to how a module is \emph{FP-injective} if and only if it is \emph{absolutely pure}. The following summarizes the first results on these modules. Note that direct sums of absolutely clean modules are again absolutely clean by the same exact argument we gave above.



\begin{proposition}\label{prop-abs-clean}
Let $R$ be any ring.  Then the absolutely clean modules are closed under direct sums.
\begin{enumerate}
\item $R$ is (left) coherent if and only if the absolutely clean modules coincide with the absolutely pure modules.
\item $R$ is (left) Noetherian if and only if the absolutely clean modules coincide with the injective modules.
\end{enumerate}
\end{proposition}

The absolutely clean modules share many more of the homological properties enjoyed by injective modules over Noetherian rings and absolutely pure modules over coherent rings. In particular, using Theorem~\ref{them-cogen-by-a-set}, the set of (isomorphism representatives of) all modules of type $FP_{\infty}$ cogenerate a \emph{hereditary} complete cotorsion pair $(\class{C},\class{A})$. Moreover, there exists an infinite cardinal $\kappa$ such that each absolutely clean module is a transfinite extension of absolutely clean modules of cardinality less than $\kappa$. We refer the reader to~\cite[Props.~2.5/2.6]{bravo-gillespie-hovey}.

\subsection{Level modules and character module duality}

For a left (resp. right) $R$-module $N$, its character module is the right (resp. left) $R$-module $N^+ = \Hom_{\Z}(N,\Q)$. For Noetherian rings $R$, it is standard that $N$ is flat if and only if $N^+$ is injective and also that $N$ is injective if and only if $N^+$ is flat; see~\cite{enochs-jenda-book}. For a coherent ring $R$ it is known that $N$ is flat if and only if $N^+$ is absolutely pure and that $N$ is absolutely pure if and only if $N^+$ is flat; see~\cite{fieldhouse}. By analyzing the reasons that make this duality possible, we see that it in fact extends to non-coherent rings. But to do so, we need the following extension of the notion of flatness which, again, is based on the modules of type $FP_{\infty}$.

\begin{definition}
A left $R$-module $L$ is called \textbf{level} if $\Tor^R_1(F,L) = 0$ for all right $R$-modules $F$ of type $FP_{\infty}$.
\end{definition}

In the spirit of Proposition~\ref{prop-abs-clean}, there are interesting characterizations of coherent rings in terms of level modules.

\begin{proposition}\label{prop-level}
Let $R$ be any ring.  Then the level modules are closed under direct products.
Moreover, $R$ is (right) coherent if and only if the level (left) modules coincide with the flat modules.
\end{proposition}

\begin{proof}
Let $\{L_{\alpha}\}$ be a given collection of level (left) $R$-modules $L_{\alpha}$ and let $F$ be a (right) $R$-module of type $FP_{\infty}$. We let $P_* \rightarrow F \rightarrow 0$ denote a projective resolution of $F$ by finitely generated projectives. Then we compute:
$$\Tor^R_1(F, \prod L_{\alpha}) = H_1(P_* \otimes_R \prod L_{\alpha}) \cong H_1(\prod (P_* \otimes_R L_{\alpha}))$$ $$\cong \prod H_1 (P_* \otimes_R L_{\alpha}) = \prod \Tor^R_1(F, L_{\alpha}) = 0.$$ The product can only be pulled out of the tensor product because each $P_i$ is finitely generated projective. The computation shows that the class of all level modules is always closed under direct products.

Now flat modules are always level and the converse holds when $R$ is (right) coherent. Indeed in this case, the (right) $R$-modules of type $FP_{\infty}$ coincide with the finitely presented modules. So every (right) module $M$ is a direct limit, $M = \varinjlim F_{i}$, of $F_i$ of type $FP_{\infty}$. It follows that $\Tor^R_1(M,L) = 0$ for all (right) modules $M$ and level (left) modules $L$.

Finally, a famous theorem of Chase states that a ring is (right) coherent if and only if the (left) flat modules are closed under direct products. So if the level (left) modules coincide with the flat modules, then from what we have already shown, they are closed under products.
\end{proof}

Many other nice properties of level modules, including the fact that they are always the left half of a complete hereditary cotorsion pair, can be found in~\cite[Section~2]{bravo-gillespie-hovey}.
In short, just as we think of the absolutely clean modules as those modules possessing the same homological properties as injective modules over Noetherian rings, we think of the level modules as the modules possessing the same properties as flat modules over coherent rings. Moreover, the following is proved in~\cite[Theorem~2.10]{bravo-gillespie-hovey}.

\begin{theorem}\label{them-character duality}
Let $R$ be a ring. A left (resp. right) $R$-module $L$ is level if and only if the right (resp. left) $R$-module $L^+$ is absolutely clean, and, a left (resp. right) $R$-module $A$ is absolutely clean if and only if $A^+$ is level.
\end{theorem}

\subsection{Gorenstein AC-injective and Gorenstein AC-projective modules}

We now return to the original goal of this section: Describing a generalization of the stable module category St($R$) of a quasi-Frobenius ring $R$. The generalization to Gorenstein rings, described in Section~\ref{section-stable module cat}, depends on the notion of Gorenstein injective and Gorenstein projective modules. An $R$-module $M$ is called \textbf{Gorenstein injective} if there exists an exact complex of injectives $$\cdots \xrightarrow{} I_1 \xrightarrow{} I_0 \xrightarrow{} I^0 \xrightarrow{} I^1 \xrightarrow{} \cdots$$ with $M = \ker{(I^0 \xrightarrow{} I^1)}$, which remains exact after applying $\text{Hom}_{R}(J,-)$ for any injective module $J$. Dualizing, one obtains the definition of \textbf{Gorenstein projective} modules. From the abelian model category point of view, generalizing the stable module category of a ring boils down to the question of when we have Hovey triples $\mathcal{S}^{inj} = (All, \class{W}_{\textnormal{inj}}, Goren \ injectives)$ and $\mathcal{S}^{proj} = (Goren \ projectives,\class{W}_{\textnormal{proj}}, All)$ in $R$-Mod. The answers are not known in full generality but it is shown in~\cite{bravo-gillespie-hovey} that this does hold for the following refinement of the Gorenstein injective and Gorenstein projective modules.

\begin{definition}\label{def-AC}
Let $R$ be any ring and $M$ an $R$-module.
\begin{enumerate}
\item $M$ is called \textbf{Gorenstein AC-injective} if there exists an exact complex of injectives $$\cdots \xrightarrow{} I_1 \xrightarrow{} I_0 \xrightarrow{} I^0 \xrightarrow{} I^1 \xrightarrow{} \cdots$$ with $M = \ker{(I^0 \xrightarrow{} I^1)}$, which remains exact after applying $\text{Hom}_{R}(A,-)$ for any absolutely clean module $A$.
\item $M$ is called \textbf{Gorenstein AC-projective} if there exists an exact complex of projectives $$\cdots \xrightarrow{} P_1 \xrightarrow{} P_0 \xrightarrow{} P^0 \xrightarrow{} P^1 \xrightarrow{} \cdots$$ with $M = \ker{(P^0 \xrightarrow{} P^1)}$, which remains exact after applying $\text{Hom}_{R}(-,L)$ for any level module $L$.
\end{enumerate}
\end{definition}

We have injective $\implies$ Gorenstein AC-injective $\implies$ Gorenstein injective, and similar for the projectives.
For Noetherian rings, the Gorenstein AC-injective modules coincide with the usual Gorenstein injectives. The Gorenstein AC-projective modules are a bit more subtle. They do coincide with the usual Gorenstein projectives whenever $R$ is Noetherian and has a dualizing complex. But it is not because the ring is Noetherian; it is because the level modules, which in this case are the flat modules, each have finite projective dimension. It is the existence of a dualizing complex which forces this~\cite{jorgensen-finite flat dimension}. Indeed the Gorenstein AC-projectives coincide with the traditional Gorenstein projectives for \emph{any} ring in which all level modules have finite projective dimension. Finally, when $R$ is coherent, the Gorenstein AC-injective (resp. Gorenstein AC-projective) modules coincide exactly with the Ding injective (resp. Ding projective) modules of~\cite{gillespie-Ding-Chen rings,Ding projective}. A main result of~\cite{bravo-gillespie-hovey} is the following, which appears there as Theorem~5.5/Proposition~5.10 and Theorem~8.5/Proposition~8.10.

\begin{theorem}[\cite{bravo-gillespie-hovey}]\label{them-Hovey-Bravo paper}
Let $R$ be any ring. Denote the class of Gorenstein AC-injective modules by $\class{GI}$ and the class of Gorenstein AC-projective modules by $\class{GP}$.
\begin{enumerate}
\item $\mathcal{S}^{inj} = (All, \class{W}_{\textnormal{inj}},\class{GI})$ is a hereditary Hovey triple, where $\class{W}_{\textnormal{inj}} = \leftperp{\class{GI}}$.
\item $\mathcal{S}^{proj} = (\class{GP},\class{W}_{\textnormal{proj}}, All)$ is a hereditary Hovey triple, where $\class{W}_{\textnormal{proj}} = \rightperp{\class{GP}}$.
\end{enumerate}
Each is cogenerated by a set. This means the corresponding triangulated homotopy categories are well generated in the sense of Neeman~\cite{neeman-book}.
\end{theorem}

Note that $\textnormal{Ho}(\mathcal{S}^{inj}) \cong \class{GI}/\sim\ $,
where $f \sim g$ if and only if $g-f$ factors through an injective module. Similarly, $\textnormal{Ho}(\mathcal{S}^{proj}) \cong \class{GP}/\sim\ $, where $f \sim g$ if and only if $g-f$ factors through a projective module. We call $\textnormal{Ho}(\mathcal{S}^{inj})$ the \textbf{injective stable module category} of $R$ and $\textnormal{Ho}(\mathcal{S}^{proj})$ the \textbf{projective stable module category} of $R$. The canonical functor $\gamma : R\textnormal{-Mod} \xrightarrow{} \textnormal{Ho}(\mathcal{S}^{inj})$ preserves coproducts, and takes all projective and absolutely clean modules to zero. It is also exact in the sense that it takes short exact sequences to exact triangles. Similarly, there is a canonical product-preserving functor $R\textnormal{-Mod} \xrightarrow{} \textnormal{Ho}(\mathcal{S}^{proj})$ which takes all injective and level modules to zero. If $R$ is a Gorenstein ring, then  $\class{W}_{\textnormal{inj}} = \class{W} = \class{W}_{\textnormal{proj}}$ is exactly the class of all modules of finite injective dimension, equivalently, of finite projective dimension.

A main difficultly in proving Theorem~\ref{them-Hovey-Bravo paper} is finding a set of modules $\class{S}$, which cogenerates  the cotorsion pair $(\class{GP}, \class{W}_{\textnormal{proj}})$. The idea is to show that any \emph{firmly acyclic} complex of projectives is ``built up from'' (technically, a \emph{transfinite extension} of) firmly acyclic complexes belonging to some fixed set, not a proper class. By a \textbf{firmly acyclic complex of projectives} we mean one that appears in the above Definition~\ref{def-AC}. That is, an exact complex of projectives $\cdots \xrightarrow{} P_1 \xrightarrow{} P_0 \xrightarrow{} P^0 \xrightarrow{} P^1 \xrightarrow{} \cdots$ which remains exact after applying $\text{Hom}_{R}(-,L)$ for all level $L$. Dealing directly with the  $\text{Hom}_{R}(-,L)$ condition seems hopeless. The problem becomes doable, although still quite technical, after the following simplification. Its proof is based on Theorem~\ref{them-character duality} and is a generalization of the result proved by Murfet and Salarian in the Noetherian case.

\begin{theorem}[\cite{murfet-salarian},\cite{bravo-gillespie-hovey}] Let $C$ be a chain complex of projective modules. Then $\text{Hom}_{R}(C,L)$ is exact for any level (left) $R$-module $L$ if and only if $A \tensor_R C$ is exact for any absolutely clean (right) $R$-module $A$.
\end{theorem}

As pointed out in Section~\ref{sec-absclean}, there is a set of absolutely clean modules from which all others are ``built up''. It implies that there is a single absolutely clean modules $A'$, for which a complex $C$ of projectives is firmly acyclic if and only if $A' \tensor_R C$ remains exact. This is the key to cogenerating $(\class{GP}, \class{W}_{\textnormal{proj}})$ by a set, proving it is a complete cotorsion pair.


\providecommand{\bysame}{\leavevmode\hbox to3em{\hrulefill}\thinspace}
\providecommand{\MR}{\relax\ifhmode\unskip\space\fi MR }
\providecommand{\MRhref}[2]{%
  \href{http://www.ams.org/mathscinet-getitem?mr=#1}{#2}
}
\providecommand{\href}[2]{#2}

\end{document}